\documentclass[12pt,reqno]{amsart}

\usepackage{enumerate}
\usepackage{mathrsfs,hyperref}
\usepackage{amsfonts}
\usepackage{epsfig,amsfonts,amsmath}
\usepackage{graphicx}
\usepackage{mathabx}
\usepackage{amsthm}
\usepackage{textcomp}
\usepackage{amsmath,amssymb}
\usepackage{fancyhdr}
\usepackage{booktabs}
\usepackage{makecell}
\usepackage{tabularx}
\usepackage{mathrsfs}
\usepackage{ragged2e}
\usepackage{bm}
\usepackage{float}
\usepackage{color}
\usepackage{dynkin-diagrams}
\usepackage{indentfirst}
\usepackage{blkarray}

\usepackage{geometry}

\title{On Valency Problems of Saxl Graphs}

\author{Jiyong Chen}
\address{J. Chen, School of Mathematical Sciences, Xiamen University, Xiamen 361005, P.R. China}
\email{cjy1988@pku.edu.cn}

\author{Hong Yi Huang}
\address{H.Y. Huang, School of Mathematics, University of Bristol, Bristol BS8 1UG, UK}
\email{hy.huang@bristol.ac.uk}

\date{\today}

\newcommand{\Mod}[1]{\ (\mathrm{mod}\ #1)}

\newcommand{\Out}{\operatorname{Out}}
\newcommand{\soc}{\operatorname{soc}}
\newcommand{\PSL}{\operatorname{PSL}}
\newcommand{\GL}{\operatorname{GL}}
\newcommand{\PGL}{\operatorname{PGL}}

\newcommand{\PGammaL}{\operatorname{P\Gamma L}}
\newcommand{\val}{\operatorname{val}}

\newcommand{\PSU}{\operatorname{PSU}}

\newcommand{\M}{\operatorname{M}}

\newcommand{\AGL}{\operatorname{AGL}}

\newcommand{\LL}{\operatorname{L}}
\newcommand{\UU}{\operatorname{U}}
\newcommand{\PGU}{\operatorname{PGU}}
\newcommand{\PGammaU}{\operatorname{P\Gamma U}}
\newcommand{\GU}{\operatorname{GU}}

\def\cali{\mathcal{I}}

\theoremstyle{plain}

\newtheorem{theorem}{Theorem}
\newtheorem{proposition}{Proposition}[section]
\newtheorem{lemma}[proposition]{Lemma}
\newtheorem{corollary}[proposition]{Corollary}
\newtheorem{thm}[proposition]{Theorem}

\newtheorem{problem}{Problem}
\newtheorem{prob}{Problem}

\theoremstyle{definition}

\newtheorem{definition}[proposition]{Definition}
\newtheorem{remark}[proposition]{Remark}
\newtheorem{example}[proposition]{Example}
\newtheorem{rmk}{Remark}

\begin{document}
\keywords{Bases, Saxl Graphs, Valencies, Subdegrees, Primitive Groups}
\maketitle

\begin{abstract}
	Let $G$ be a permutation group on a set $\Omega$ and recall that a base for $G$ is a subset of $\Omega$ such that its pointwise stabiliser is trivial. In a recent paper, Burness and Giudici introduced the Saxl graph of $G$, denoted $\Sigma(G)$, with vertex set $\Omega$ and two vertices adjacent if and only if they form a base for $G$. If $G$ is transitive, then $\Sigma(G)$ is vertex-transitive and it is natural to consider its valency (which we refer to as the valency of $G$). In this paper we present a general method for computing the valency of any finite transitive group and we use it to calculate the exact valency of every primitive group with stabiliser a Frobenius group with cyclic kernel. As an application, we calculate the valency of every almost simple primitive group with an alternating socle and soluble stabiliser and we use this to extend results of Burness and Giudici on almost simple primitive groups with prime-power or odd valency.
\end{abstract}

% \providecommand{\keywords}[1]{\textbf{[Keywords]} #1}

% \keywords{Bases, Saxl graphs, Valencies, Subdegrees, Primitive groups}

%\begin{keywords}
%	Bases, Saxl graphs, Valencies, Subdegrees, Primitive groups
%\end{keywords}

\section{Introduction}

Let $G$ be a finite permutation group on a set $\Omega$. A \textit{base} for $G$ is a subset of $\Omega$ such that its pointwise stabiliser is trivial. The \textit{base size} of $G$, denoted by $b(G)$, is the minimal size of a base for $G$. This is a classical notion in permutation group theory and bounds on the base sizes of finite permutation groups have been studied since the nineteenth century, finding a wide range of applications. For example, see \cite{BaseMetric} for details of the relationship between the base sizes of a group and the metric dimension of a graph, and \cite[Section 4]{PermutaionGroupAlgorithms} for the application of bases in the computational study of finite groups.

In more recent years, there has been significant interest in determining bounds on the base sizes of finite primitive groups, and almost simple groups in particular (recall that a group $G$ is \emph{almost simple} if $T \leqslant G \leqslant {\rm Aut}(T)$ for some non-abelian simple group $T$, which is the socle of $G$). Some of this interest has been partly motivated by a well known conjecture of Cameron and Kantor \cite{Cameron_open_problem,Cameron-Kantor}, which asserts that there is an absolute constant $c$ such that $b(G) \leqslant c$ for every non-standard almost simple primitive group $G$ (we refer the reader to \cite{BurnessBasesAlmostSimple} for the definition of a \emph{non-standard} group). This conjecture was proved by Liebeck and Shalev \cite{LiebeckShalev1999} using probabilistic methods and it is now known that $c=7$ is the optimal constant (in fact, $b(G) = 7$ if and only if $G$ is the Mathieu group ${\rm M}_{24}$ acting on $24$ points); see the sequence of papers \cite{BurnessBasesClassical,BurnessBasesSymmetric,Tim_Cameron_conjecture,BurnessBasesSporadic} by Burness et al. Furthermore, almost simple primitive groups with $b(G)=6$ have been determined in \cite[Theorem 1]{BurnessBasesAlmostSimple}. If $G$ is a soluble primitive group, then a theorem of Seress \cite{Seress1996} shows that $b(G)\leqslant 4$ and this has very recently been extended by Burness \cite{Burness2020base}, who has established the bound $b(G)\leqslant 5$ for any finite primitive group $G$ with a soluble point stabiliser (in both cases, the bounds are best possible). In addition,  \cite[Theorem 2]{Burness2020base} gives the exact base size for every almost simple primitive group with a soluble stabiliser.

There has been a special interest in studying the permutation groups with base size $2$. Indeed, a programme of research initiated by Saxl in the 1990s seeks to determine all the primitive groups $G$ with $b(G) = 2$. In \cite{TimSaxlGraph}, Burness and Giudici introduced the \textit{Saxl graph} of a permutation group $G \leqslant {\rm Sym}(\Omega)$, denoted by $\Sigma(G)$, as a tool for studying these groups. Here the vertex set is $\Omega$ and two vertices are adjacent if and only if they form a base for $G$. Recall that an \textit{orbital graph} of $G$ is a graph with vertices $\Omega$ and $(\alpha,\beta)$ is a directed edge if it is contained in a fixed orbital (an orbit of the associated action of $G$ on $\Omega\times\Omega$) of $G$. Indeed, $\Sigma(G)$ is the union of all regular orbital graphs of $G$ (a regular orbital is an orbital on which $G$ acts regularly). We refer the reader to \cite[Lemma 2.1]{TimSaxlGraph} for the basic properties of $\Sigma(G)$.

Now assume $G \leqslant {\rm Sym}(\Omega)$ is a finite transitive group with point stabiliser $H$, in which case $\Sigma(G)$ is vertex-transitive. Let $\val(G,H)$ be the valency of $\Sigma(G)$ and observe that $\val(G,H) = r|H|$, where $r$ is the number of regular orbits of $H$ on $\Omega$. It is easy to see that if $\val(G,H)> \frac{1}{2}|\Omega|$, then $\Sigma(G)$ is connected with diameter at most $2$. The \textit{Burness-Giudici conjecture} from \cite{TimSaxlGraph} asserts that $\Sigma(G)$ has diameter at most $2$ for every finite primitive group $G$ with $b(G) = 2$ and this provides further motivation for investigating $\val(G,H)$ in this paper. We refer the reader to \cite{BH_Saxl,BG_PSL,LP_Saxl} for some recent work on this conjecture.
%\textbf{Do your results settle any new cases of the conjecture?}

We will see another application of $\val(G,H)$ in the following remark.
	
\begin{rmk}
	Another motivation for determining $\val(G,H)$ comes from the study of the bases for primitive groups of \textit{product type} (we refer the reader to \cite[p.391, III(b)]{O'Nan-Scott} for the definition of product type groups). For example, let $X\le \operatorname{Sym}(\Gamma)$ be a base-two primitive group with stabiliser $Y$, and $G = X\wr P$ acting on the Cartesian product $\Omega=\Gamma^k$ with its product action, where $P\le S_k$ is transitive. Then \cite[Theorem 2.13]{BaseMetric} implies that $b(G) = 2$ if and only if $\val(X,Y)/|Y|$ is at least the distinguishing number of $P$ (see also \cite[Corollary 2.9]{TimSaxlGraph}). Here the \textit{distinguishing number} of $P$ is the smallest size of a partition of $\{1,\dots,k\}$ such that only the identity element of $P$ fixes all the parts of the partition. In particular, $b(X\wr S_k) = 2$ if and only if $\val(X,Y)\ge k|Y|$ since the distinguishing number of $S_k$ is $k$.
\end{rmk}

Recall that a group $H$ is said to be \textit{Frobenius} if there exists a non-trivial proper subgroup $L< H$ such that $L\cap L^h=1$ for all $h\in H\setminus L$. The subgroup $L$ is called the \textit{Frobenius complement} of $H$. The \textit{Frobenius kernel} $K$ is the subgroup comprising the identity element and those elements that are not in any conjugate of $L$. A well known result \cite{Frobenius1901} states that $H=K{:}L$ is a split extension. It is also easy to show that if $K$ is cyclic, then $L$ is also cyclic.

Our first main result gives an explicit formula for $\val(G,H)$ in the case where $G$ is primitive and $H$ is a Frobenius group with cyclic kernel (see Section \ref{section Frobenius} for the proof). Here, the \textit{M\"{o}bius function} $\mu$ is the function defined on the set of positive integers such that $\mu(k) = 0$ if $k$ is not square-free, $\mu(k) = -1$ if $k$ is square-free and has an odd number of prime factors, and $\mu(k) = 1$ otherwise.

%\textbf{May be helpful to say a few words about ``Frobenius groups" here...}

%\textbf{The first theorem should be ``short and snappy" -- I suggest the second part is omitted here as it makes the result look very technical...}

\begin{theorem}
    \label{thm Frobenius val intro}
    Suppose $G$ is a finite primitive permutation group with point stabiliser $H$, where $H=K{:}L$ is Frobenius with cyclic kernel $K$. Write $L=\langle y\rangle$ and let $\pi(L)$ be the set of divisors $d$ of $|L|$ with $d>1$. Then
    \[
    \val(G,H)=|G:H|+|K|-1+\frac{|K|}{|L|}\sum_{d \in \pi(L)} \mu(d) |N_G(\langle y^{\frac{|L|}{d}}\rangle)|,
    \]
where $\mu$ is the M\"{o}bius function. 
%Moreover, all the other non-trivial subdegrees are $d|K|$ for proper divisors $d$ of $|L|$, with multiplicities
%\begin{equation}
%\label{equation subdegree Frobenius}
%m\left(G,H,d|K|\right)=
%\begin{cases}
%\frac{1}{d|L|}\sum_{e\mid d}\mu(e)|N_G(\langle y^{\frac{d}{e}}\rangle)| & 1<d<|L|\\
%\frac{1}{|L|}(|N_G(L)|-|L|) & d=1.
%\end{cases}
%\end{equation}
\end{theorem}

We refer the reader to Theorem \ref{thm Frobenius val} for a more general result, which describes all the subdegrees of $G$ and their associated multiplicities. 

In order to prove Theorem \ref{thm Frobenius val intro} we introduce a general method in Section \ref{section Strategy} for computing subdegrees and their associated multiplicities of a transitive group, which is a generalisation of \cite{Subdegree}. To apply this strategy we need to determine all possible cases $H \cap H^g$ for $g \in G$. This leads us to the following problem, which may be of independent interest.

\begin{problem}\label{conjecture normal arc stabiliser}
    Determine the primitive permutation groups $G\le \operatorname{Sym}(\Omega)$ such that there exists $\alpha,\beta\in\Omega$ satisfying $1\ne G_{\alpha\beta}\lhd G_\alpha$.
\end{problem}

The problem was initially stated by Cameron in \cite{CameronProblem}, where he conjectured that there is no primitive permutation group satisfying the condition in Problem \ref{conjecture normal arc stabiliser}. It is straightforward to see that there is no affine primitive group satisfying the condition. We refer the reader to Konygin's work \cite{Konygin2008,Konygin2011,Konygin2013,Konygin2015,Konygin2019} on this problem when $G$ is almost simple or an associated product type primitive group. In particular, no example arises in the case when $G$ has soluble point stabilisers (see \cite[Proposition 8]{Konygin2008}). Recently, however, Spiga \cite[Theorem 1.4]{Spiga} first found an example satisfying the condition in Problem \ref{conjecture normal arc stabiliser}, which is a primitive group of diagonal type (see \cite[Section 5]{Spiga} for the construction). We refer the reader to Remarks \ref{remark conjecture 1 edge stabiliser} and \ref{remark conjecture 1 Richard Lyons} for further remarks to this problem.

%That is, a non-trivial normal subgroup of the point stabiliser of a primitive permutation group may never be an arc stabiliser. Some observations of Conjecture \ref{conjecture normal arc stabiliser} are given in Remarks \ref{remark conjecture 1 edge stabiliser} and \ref{remark conjecture 1 Richard Lyons}.

Theorem \ref{thm Frobenius val intro} can be applied to various problems. Our first application
concerns the almost simple primitive groups with socle an alternating group. Let $G$ be an almost simple primitive group with $\soc(G)=A_n$ and soluble stabiliser $H$. Note that $\val(G,H)=0$ if $b(G)>2$, and those groups with $b(G)=2$ are classified in \cite{Burness2020base}. In the following theorem, $\mu$ denotes the M\"obius function and $\phi$ denotes the Euler totient function.

\begin{theorem}
\label{thm table val almost simple soluble stabiliser alternating}
	Let $G$ be an almost simple primitive group with socle $A_n$ and soluble stabiliser $H$. If $b(G)=2$, then $(G,H,\val(G,H))$ is listed in Table \ref{table val almost simple soluble stabiliser alternating}, where
	\begin{equation}\label{equation val Ap}
	\val(A_p,\mathbb{Z}_p{:}\mathbb{Z}_{(p-1)/2})=(p-2)!+p-1+p\sum_{1\ne d\mid\frac{p-1}{2}}\mu(d)\phi(d)d^{\frac{p-1}{d}-1}\left(\frac{p-1}{d}-1\right)!
	\end{equation}
	and
	\begin{equation}\label{equation val Sp}
	\val(S_p,\AGL_1(p))=(p-2)!+p-1+p\sum_{1\ne d\mid (p-1)}\mu(d)\phi(d)d^{\frac{p-1}{d}-1}\left(\frac{p-1}{d}-1\right)!.
	\end{equation}
	%Moreover, the multiplicities of subdegrees in the only two infinite cases are given in (\ref{equation subdegree Sp}) and (\ref{equation subdegree Ap}).

\end{theorem}

\begin{table}[ht]
	\centering
	\begin{tabular}{llll}
		\Xhline{2pt}
		$G$&$H$&$\val(G,H)$&Conditions\\
		\Xhline{1pt}
		$A_5$&$S_3$&$6$&\\
		$\M_{10}$&$\AGL_1(5)$&$20$&\\
		$\M_{10}$&$8{:}2$&$32$&\\
		$\PGL_2(9)$&$D_{16}$&$16$&\\
		$A_9$&$\operatorname{ASL}_2(3)$&$432$&\\
		$A_p$&$\mathbb{Z}_p{:}\mathbb{Z}_{(p-1)/2}$&see (\ref{equation val Ap})&$p> 5$ is a prime, $p\ne 7,11,17,23$\\
		$S_p$&$\AGL_1(p)$&see (\ref{equation val Sp})&$p> 5$ is a prime\\
		\Xhline{2pt}
	\end{tabular}
	\caption{The cases $(G,H,\val(G,H))$ in Theorem \ref{thm table val almost simple soluble stabiliser alternating}}
	\label{table val almost simple soluble stabiliser alternating}
	
\end{table}

Next we turn to the primitive groups with prime-power valency Saxl graphs. The following result builds on \cite[Proposition 3.1]{TimSaxlGraph}, which describes the transitive groups $G$ with $\val(G,H)$ a prime. The proof of the following result, given in Section \ref{section prime power}, is based on the classification of almost simple primitive groups with stabiliser of prime-power order (see Proposition \ref{proposition almost simple stabiliser p-group}). Recall that the Johnson graph $J(n,k)$ is a graph with vertices the set of $k$-subsets of an $n$-element set, and two vertices are adjacent if they contain exactly $(k-1)$ elements in common.
 
\begin{theorem}
	Let $G$ be an almost simple primitive group with stabiliser $H$. Then $\val(G,H)$ is a prime power if and only if one of the following holds:
	\begin{enumerate}[{\indent\rm (i)}]
		\item $(G,H)=(\M_{10},8{:}2)$ with $\val(G,H)=32$.
		\item $(G,H)=(\PGL_2(q),D_{2(q-1)})$, where $q\ge 17$ is a Fermat prime or $q=9$, $\Sigma(G)$ is isomorphic to $J(q+1,2)$ and $\val(G,H)=2(q-1)$.
	\end{enumerate}
\label{thm classification prime-power}
\end{theorem}

Similarly, for groups with odd valency we obtain Theorem \ref{thm odd valency rewrite}, which extends \cite[Proposition 3.2]{TimSaxlGraph} (in particular, Theorem \ref{thm table val almost simple soluble stabiliser alternating} above shows that case (iii) in \cite[Proposition 3.2]{TimSaxlGraph} does not arise). The notation for classical groups follows  \cite{KleidmanLiebeckClassicalGroups}, where $\epsilon=+$ or $-$ indicates the linear or unitary case, respectively. For the proof of Theorem \ref{thm odd valency rewrite} we refer the reader to Section \ref{section odd valencies}.

\begin{theorem}
    Let $G$ be an almost simple primitive group with stabiliser $H$ and $b(G)=2$. Then $\val(G,H)$ is odd only if one of the following holds:
    \begin{enumerate}[{\indent\rm (i)}]
        \item $G=\M_{23}$ and $H=23{:}11$.
        \item $G=\LL_r^\epsilon(q).O\le \PGammaL_r^\epsilon(q)$ and $H=\mathbb{Z}_a{:}\mathbb{Z}_r.O$, where $a=\frac{q^r-\epsilon}{(q-\epsilon)(r,q-\epsilon)}$, $r$ is an odd prime and $O\le \Out(\LL_r^\epsilon(q))$ has odd order. In addition, $G$ is not a subgroup of $\PGL_r^\epsilon(q)$.
    \end{enumerate}
    \label{thm odd valency rewrite}
\end{theorem}

However, it is not known if there are any genuine examples satisfying the conditions in part (ii).

% \begin{conjecture}
% 	\label{conjecture odd valency}
% 	Let $G$ be an almost simple primitive group with stabiliser $H$. Then $\val(G,H)$ is odd if and only if $G = {\rm M}_{23}$ and $H = 23{:}11$.
% \end{conjecture}

%All of the above theorems on almost simple primitive groups are based on the classification of almost simple primitive groups with soluble stabilisers in \cite{SolubleStabiliser}.

\subsection*{Notation}

We will denote by $\phi$ the Euler totient function and we denote the cyclic group of order $n$ by $\mathbb{Z}_n$. By $(a,b)$ we mean the greatest common divisor $\gcd(a,b)$ of two integers $a$ and $b$. Our notation for classical groups 
follows \cite{KleidmanLiebeckClassicalGroups}. For example, we use $\LL_n(q)$ or $\LL_n^+(q)$ to denote $\PSL_n(q)$, and sometimes $\PGL_n^+(q)$ and $\PGammaL_n^+(q)$ to represent $\PGL_n(q)$ and $\PGammaL_n(q)$ respectively. Similarly, for unitary groups we write $\UU_n(q)$ or $\LL_n^-(q)$ to represent $\PSU_n(q)$ and we also use $\PGL_n^-(q)=\PGU_n(q)$ and $\PGammaL_n^-(q)=\PGammaU_n(q)$.

\subsection*{Acknowledgements}

This work was partially supported by the National Natural Science Foundation of China (Grant No. 11931005) and the Fundamental Research Funds for the Central Universities (Grant No. 20720210036). The second author is supported by China Scholarship Council for his doctoral studies at the University of Bristol.

Both authors thank Tim Burness, Cai Heng Li and Binzhou Xia for their helpful discussions. They deeply thank Southern University of Science and Technology (SUSTech) for their support and hospitality when some of the work on this paper was undertaken. They also thank Derek Holt, Richard Lyons, Geoffrey Robinson and Gabriel Verret for their comments on Problem 1 on MathOverflow (question 372398 posted by the second author).

\section{Preliminaries}

\subsection{Partially ordered sets and M\"obius functions}

At the beginning of this section, we prove two useful lemmas concerning the inversion formula on a partially ordered set. Let $(P,\le)$ be a finite partially ordered set.
\begin{lemma}\label{lemma existence of inversion formula on poset}
	Let $c$ be a function from $P\times P$ to $\mathbb{C}$, such that 
	\[c(p_1,p_2)=\begin{cases}
	 	1	&	\mbox{if $p_1=p_2$, }\\
	 	0	&	\mbox{if $p_1,p_2$ are not comparable or $p_1>p_2$. }\\
	\end{cases}\] 
	Then there is a unique function $d:P\times P\rightarrow \mathbb{C}$ such that 
	\[\sum_{p\in P}c(p_1,p)d(p,p_2)=\sum_{p\in P}d(p_1,p)c(p,p_2)=\begin{cases}
	1 &\mbox{ if $p_1=p_2$,}\\
	0 &\mbox{ otherwise.}\\
	\end{cases}\]
\end{lemma}

\begin{proof}
Since $(P,\le)$ is finite, there is a linear extension $(P,\preceq)$ of it. That is, $(P,\preceq)$ is a totally ordered set, and for any $p_1,p_2\in P$, $p_1\preceq p_2$ whenever $p_1\le p_2$. Suppose that $n=|P|$. Label the elements of $P$ by $p_1,p_2,\dots, p_n$ such that $p_i\preceq p_j$ whenever $i\le j$. Let $C=[c(p_i,p_j)]_{n\times n}$ be the matrix with $c(p_i,p_j)$ being the entry in the $i$-th row and $j$-th column.  By the hypothesis on $c$, we have $c(p_i,p_i)=1$ for $1\le i \le n$ and $c(p_i,p_j)=0$ when $n\ge i > j\ge 1$. It follows that the matrix $C$ is an upper triangular matrix with diagonal entries all equal to $1$. In particular, $C$ is invertible. Suppose $D=[d_{ij}]_{n\times n}=C^{-1}$ and let $d(p_i,p_j)=d_{ij}$ be the $(i,j)$-entry of $D$. It is obvious that the function $d$ satisfies those equalities listed in this lemma which is uniquely determined by $c$. 
\end{proof}

\begin{lemma}\label{inversion formula on poset}
	Let $c,d$ be two functions as defined in Lemma~\ref{lemma existence of inversion formula on poset}. Suppose that $f,g:P\rightarrow \mathbb{C}$ are two functions on $P$ such that 
		\[f(q)=\sum_{p\in P}c(q,p)g(p)\]
	for each $q\in P$. Then for each $p\in P$,
	\[g(p)=\sum_{q\in P}d(p,q)f(q).\]
% 	\begin{enumerate}[\rm (i)]
% 		\item For any $p_1,p_2\in P$
% 		\[\sum_{p\in P}d(p_1,p)c(p,p_2)=\begin{cases}
% 			1 &\mbox{ if $p_1=p_2$;}\\
% 			0 &\mbox{ otherwise.}\\
% 			\end{cases}\]
% 		\item 
% 		\item Suppose $c(p_1,p_2)=1$ for each pair $p_1\le p_2$. That is, $c$ is the zeta function. Then $d=\mu_{P^{*}}$ is the M\"obius function of the dual partially ordered set $P^{*}$ of $P$.
% 	\end{enumerate}
	
\end{lemma}

\begin{proof}
% 	Part~(i) can be obtained directly from the proof of Lemma~\ref{lemma existence of inversion formula on poset}.

	For any $p\in P$, 
	\begin{align*}
		\sum_{q\in P}d(p,q)f(q) 
		& =\sum_{q\in P}d(p,q)\left(\sum_{x\in P}c(q,x)g(x)\right)\\ 
		& =\sum_{x\in P}\left(\sum_{q\in P}d(p,q)c(q,x)g(x)\right)\\ 
		& =\sum_{x\in P}g(x)\left(\sum_{q\in P}d(p,q)c(q,x)\right).
		% & \sum_{x\in P}g(x)\sum_{q\in P}d(p,q)c(q,x)\\ 
	\end{align*}
	By Lemma~\ref{lemma existence of inversion formula on poset}, the only nonzero term on the right side is $g(p)$, which completes the proof.
\end{proof}

\begin{remark}
    If we also assume that $c(p_1,p_2)=1$ for each pair $(p_1,p_2)$ with $p_1\le p_2$, then $c,d$ are indeed the zeta function and M\"obius function on the order-dual $P^{\downarrow}$ of $P$, respectively. See \cite[Sections 1.1 and 3.1]{Rota} for more details.
\end{remark}

\subsection{Frobenius groups with cyclic kernel}

Let $H=K{:}L$ be a Frobenius group with Frobenius kernel $K$ and Frobenius complement $L$. Then it is well known that $(|K|,|L|)=1$ (see for example \cite[Theorem 12.6.1]{Scott}). Suppose $K$ is cyclic. We list the following basic and well-known properties of $H$, which will be useful later. The proof is straightforward and we refer the reader to \cite[Section 12.6]{Scott} for more properties.

\begin{lemma}
    \label{lemma Frobenius}
    In terms of the above notation, the following holds.
    \begin{enumerate}[{\indent\rm (i)}]
        \item The Frobenius complement $L$ is cyclic.
        \item For any non-trivial subgroup $B$ of $L$, $N_H(B)=L$.
        \item For any non-trivial subgroup $B$ of $H$ with $B\cap K=1$, there exists a unique $k\in K$ such that $B^k\leq L$.
    \end{enumerate}
\end{lemma}

\section{Our Strategy}

\label{section Strategy}

Throughout this section, we assume $G$ is a transitive permutation group on a set $\Omega$ with stabiliser $H$. 
Suppose that $H=G_\alpha$ is the stabiliser of the point $\alpha\in \Omega$. 

\begin{definition}
	The \textit{multiplicity} of a subdegree $n$, denoted by $m(G,H,n)$, is the number of suborbits of $G$ of length $n$.
\end{definition}

In particular, we have $\val(G,H)=|H|\cdot m(G,H,|H|)$ and $m(G,H,1)=1$.

In order to calculate $\val(G,H)$ and the general $m(G,H,n)$, we need to determine the cardinality of the set $\{g\in G\mid G_\alpha\cap G_{\alpha^g}=A\}$ for any subgroup $A$ of $H$.

\subsection{Basic enumeration}

Write $$\delta_H^G(A):=\{g\in G\mid G_\alpha\cap G_{\alpha^g}=A\}=\{g\in G\mid H\cap H^g=A\}.$$
If $G$ and $H$ are clear from the context, we will write $\delta_H(A)$ or $\delta(A)$ for short.
In particular, $\delta(1)$ is the set of $g\in G$ such that $\alpha^g$ is adjacent to $\alpha$ in $\Sigma(G)$. We have the following observations on $\delta(A)$.
\begin{lemma}\label{lemma delta}
	With the notation above, the following statements hold:
	\begin{enumerate}[{\rm (i)}]
		\item $G=\bigcup_{A\le H}\delta(A)$;
		\item $\val(G,H)=\frac{|\delta(1) |}{|H|}$;
		\item For any $h\in H$ and any $A\le H$,  $\delta(A^h)=\delta(A)h$.
	\end{enumerate}
\end{lemma}

\begin{proof}
    Parts (i) and (ii) follow immediately from the definitions. Let $g\in\delta(A^h)$. Then $H\cap H^g=A^h$ and so $H\cap H^{gh^{-1}}=A$. Conversely, if $g\in \delta(A)$ then $H\cap H^{gh}=A^h$, which implies that $gh\in \delta(A^h)$. Hence, part (iii) holds. 
\end{proof}

Note that if $A< H$ and $\delta(A)$ is non-empty, then $A$ is an arc stabiliser of some orbital graph of $G$. In general, however, for a fixed subgroup $A$ of $H$, it is not easy to determine whether $\delta(A)$ is empty or not. For example, if $G$ is primitive, it is not known if there exists a non-trivial proper normal subgroup $A$ of $H$ such that $\delta(A)$ is non-empty. Indeed, checking the relevant database in {\sc{Magma}} \cite{Magma}, we determine that there is no example of such a primitive group with degree at most 4095. This leads us naturally to the following problem (as stated in the introduction).

\begin{prob}
	Determine the primitive permutation groups $G\le \operatorname{Sym}(\Omega)$ such that there exists $\alpha,\beta\in\Omega$ satisfying $1\ne G_{\alpha\beta}\lhd G_\alpha$.
\end{prob}

Here are two remarks concerning Problem \ref{conjecture normal arc stabiliser}.

\begin{remark}\label{remark conjecture 1 edge stabiliser}
Write $H\cap H^g$ as the arc stabiliser $G_{(\alpha,\alpha^g)}$ of the orbital graph associated to the orbital $(\alpha,\alpha^g)$ for some $g\notin G_\alpha=H$. Suppose $1\ne G_{(\alpha,\alpha^g)}\lhd G_\alpha$. If the edge stabiliser $G_{\{\alpha,\alpha^g\}}$ is strictly larger then $G_{(\alpha,\alpha^g)}$ (that is, the orbital is self-paired and the associated orbital graph is undirected), then there exists $x\in G$ such that $\alpha^x=\alpha^g$ and $\alpha^{gx}=\alpha$. This yields $x\in N_G(G_{(\alpha,\alpha^g)})=G_\alpha$. Hence, $\alpha=\alpha^x=\alpha^g$, a contradiction. This verifies Problem \ref{conjecture normal arc stabiliser} when $G_{(\alpha,\alpha^g)}$ is a proper subgroup of $G_{\{\alpha,\alpha^g\}}$.
\end{remark}

\begin{remark}\label{remark conjecture 1 Richard Lyons}
	Suppose $G\le \operatorname{Sym}(\Omega)$ is primitive and $1\ne G_{\alpha\beta}\lhd G_\alpha$ for some $\alpha,\beta\in\Omega$. Then by \cite[Lemmas 2.1 and 2.2]{SimpleSection}, $G_{\alpha}$ and $G_\alpha/G_{\alpha\beta}$ have the same simple sections (a section is a quotient group of a subgroup). In particular, $G_{\alpha}$ and $G_{\alpha}/G_{\alpha\beta}$ have the same solubility and their orders have same prime divisors.
There are some other observations on $G_{\alpha\beta}$ if $1\ne G_{\alpha\beta}\lhd G_\alpha$ given in MathOverflow (Question 372398). For example, Richard Lyons noted that $G_{\alpha\beta}$ must have even order by analysing on the generalised Fitting subgroups.
\end{remark}

With also Lemma \ref{lemma delta}(iii) in mind, we pick a subset $\mathcal{I}$ of $\{A\mid A\le H\}$ such that
\begin{equation*}
\bigcup_{A\in \mathcal{I}}\{A^h\mid h\in H\}\supseteq\{A\le H\mid \delta(A)\ne \emptyset\}.
\end{equation*}
This is a set of representatives of possible arc stabilisers $A$ up to conjugacy in $H$.

\begin{lemma}\label{lemma m(G,H,n) first observation}
	We have $$m(G,H,n)=\frac{1}{n} \sum_{A\in \mathcal{I}, |A|=\frac{|H|}{n}} \frac{|\delta(A)|}{|N_H(A)|}.$$
\end{lemma}

\begin{proof}
	First note that $A$ has exactly $|H:N_H(A)|$ distinct conjugates in $H$. Thus,
	\begin{equation*}
	m(G,H,n)=\frac{1}{n|H|}\sum_{A\le H,|A|=\frac{|H|}{n}}|\delta(A)|=\frac{1}{n|H|}\sum_{A\in\mathcal{I},|A|=\frac{|H|}{n}}|\delta(A)|\cdot|H:N_H(A)|
	\end{equation*}
	as required.
\end{proof}

To calculate the multiplicities of subdegrees as well as the valency, we need to find a way to calculate $|\delta(A)|$. For this purpose, we define some new sets. 
Let $\Delta_H^G(A)$ be the set $\{g\in G\mid H\cap H^g\ge A\}$. Again, we write $\Delta_H(A)$ or even $\Delta(A)$ for short if $G$ and $H$ are clear from the context.

\begin{lemma}\label{lemma Delta}
    Let $\mathcal{S}$ be a set of representatives of the $H$-conjugacy classes of subgroups of $H$.
    \begin{enumerate}[{\rm (i)}]
        \item $\Delta(A)=\bigcup_{A^x\in \mathcal{S}\cap A^G} HN_G(A^x)x^{-1} $;
		\item if $A^g\le H$ for some $g\in G$, then $\Delta(A^g)=\Delta(A)g$ and
		$$|\Delta(A^g)|=|\Delta(A)|=\sum_{B\in \mathcal{S}\cap A^G}\frac{|H| |N_G(B)|}{|N_H(B)|}.$$
    \end{enumerate}
\end{lemma}

\begin{proof}
	Firstly, we have
	\begin{equation*}
	\begin{aligned}
	\Delta(A)
	&=\{g\in G\mid H\cap H^g\ge A\}=\{g\in G\mid A^{g^{-1}} \le H\}\\
	&=\bigcup_{C\in A^G,\ C\le H}\{g\in G\mid A^{g^{-1}} =C\}\\
	&=\bigcup_{B\in \mathcal{S}\cap A^G}\left(\bigcup_{C\in B^H} \{g\in G\mid A^{g^{-1}} = C\}\right).\\
	% &=\biguplus_{C\in \mathcal{R}^G_H(A)}\left(\biguplus_{B\in C^H} \{g\in G| A^{g^{-1}} = B\}\right)\\
	\end{aligned}
	\end{equation*}
	For any group $B\in \mathcal{S}\cap A^G$, suppose that $B=A^{x}\le H$ for some $x\in G$. It follows that
	\begin{equation*}
	\begin{aligned}
	\bigcup_{C\in B^H} \{g\in G\mid A^{g^{-1}} = C\}
%	=&	 \biguplus_{C\in A^{xH}} \{g\in G\mid A^{g^{-1}} = {C\}\\
	=&	 \bigcup_{h \in H} \{g\in G\mid A^{g^{-1}} = A^{xh}\} \\
%	=&	 \bigcup_{h \in H} \{g\in G\mid A = A^{xhg}\}\\
	=&	 \bigcup_{h \in H} h^{-1}x^{-1}N_G(A)\\
	=&	 Hx^{-1}N_G(A)\\
%	=&	 HN_G(A^x)x^{-1}\\
	=&	 HN_G(B)x^{-1}.\\
	\end{aligned}
	\end{equation*}
	This gives part (i). Combining the above equations, we obtain that
	\begin{equation*}
	\begin{aligned}
	|\Delta(A) |
	&=\sum_{B\in \mathcal{S} \cap A^G}\left|\bigcup_{C\in B^H} \{g\in G\mid A^{g^{-1}} = C\}\right|\\
	&=\sum_{B\in \mathcal{S} \cap A^G}\left|HN_G(B)\right|.\\
	&=\sum_{B\in\mathcal{S}\cap A^G}\frac{|H||N_G(B)|}{|N_H(B)|}.
	\end{aligned}
	\end{equation*}
	Note that the right-hand-side remains unchanged if we replace $A$ by $A^g\le H$ for some $g\in G$. This completes the proof.
\end{proof}

By the virtue of Lemma~\ref{lemma Delta}, $|\Delta(A)|$ can be easily calculated if both normalisers in $G$ and $H$ are known for each subgroup of $H$.  
Note that
\begin{equation*}
\Delta(A)=\bigcup_{B\ge A}\delta(B)
\end{equation*}
is a disjoint union and so
\begin{equation*}\label{eqn Delta delta original}
|\Delta(A)|=\sum_{B\ge A}|\delta(B)|.
\end{equation*}
Now, let $P=(\{A\mid A\le H\}, \le)$ be the partially ordered set on all subgroups of $H$ with the natural inclusion relation. Let $c:P\times P\rightarrow \mathbb{C}$ be a function on $P\times P$ such that \[c(A,B)=\begin{cases}
 1& \mbox{ if $A\le B$;}\\
 0& \mbox{ otherwise.}\\
\end{cases} \]
It is obvious that $c$ satisfies the hypothesis in Lemmas~\ref{lemma existence of inversion formula on poset} and \ref{inversion formula on poset}. This implies that
\begin{equation}
    \label{eqn D in cd}
    |\Delta(A)|=\sum_{B\in P}c(A,B)|\delta(B)|
\end{equation}
and
\begin{equation*}
    |\delta(B)|=\sum_{A\in P}\mu_{P^{\downarrow}}(B,A)|\Delta(A)|,
\end{equation*}
where $\mu_{P^{\downarrow}}$ is the M\"obius function on the dual partially ordered set $P^{\downarrow}$. 
This provides a way to compute $|\delta(A)|$ for any subgroup $A\le H$.

\subsection{Reduction}

In the rest of this section, we aim to reduce the size of the partially ordered set $P$ in the calculation of $|\delta(A)|$. By Lemma~\ref{lemma delta}(iii), subgroups in the same conjugacy class in $H$ give same cardinalities of $\delta(A)$ and $\Delta(A)$. Hence, it is natural to consider the partially ordered set on the set $\mathcal{I}$. It is easy to make $\mathcal{I}$ into a partially ordered set by defining $A\preceq B$ in $\mathcal{I}$ if there exists an element $h\in H$ such that $A\le B^h$.

\begin{lemma}\label{lemma inversion formula on cali}
	With the notation above, we have the following statements.
	\begin{enumerate}[{\rm (i)}]
		\item If $\eta$ is a function defined on $\cali\times\cali$ such that $\eta(A,B)=|\{D\in B^H\mid D\ge A\}|$, then
		\begin{equation}\label{eqn eta}
			|\Delta(A)|=\sum_{B\in \mathcal{I}} \eta(A,B)|\delta(B)|.
		\end{equation}
		\item There exists a function $\mu_{\eta}:\mathcal{I}\times \mathcal{I}\rightarrow \mathbb{C}$, such that for each $B\in \cali$,
		\[ |\delta(B)|=\sum_{A\in \mathcal{I}} \mu_{\eta}(B,A)|\Delta(A)|.\]
	\end{enumerate}
\end{lemma}

\begin{proof}
	By (\ref{eqn D in cd}),
	  \begin{align*}
	  |\Delta(A)|
	  &=\sum_{B\in P}c(A,B)|\delta(B)| \\
	  &=\sum_{B\in \cali} \left(\sum_{C \in B^H }c(A,C)|\delta(C)|\right)\\
	  &=\sum_{B\in \cali} \left(\sum_{C \in B^H }c(A,C)|\delta(B)|\right)\\
	  &=\sum_{B\in \cali} \left(\sum_{C \in B^H, C\ge A }1 \right)|\delta(B)|\\
	  &=\sum_{B\in \cali} |\{B^h\mid  B^h\ge A\} |\cdot|\delta(B)|.
	  \end{align*}
	  Thus, part~(i) holds. Now, by applying Lemma~\ref{inversion formula on poset}, part~(ii) holds.
\end{proof}

By Lemma~\ref{lemma Delta}(ii), $|\Delta(A)|=|\Delta(B)|$ if $A$ and $B$ two subgroups of $H$ which are conjugate in $G$. Note that $A,B$ are not necessarily conjugate in $H$. In most cases, however, we still have $|\delta(A)|=|\delta(B)|$, which would simplify (\ref{eqn eta}). To do this, we define $A\sim B$ for $A,B\in\mathcal{I}$ if the following conditions hold (here we adopt the notation of Lemma \ref{lemma inversion formula on cali}):
\begin{enumerate}[(E1)]
	\item $B=A^g$ for some $g\in G$;
	\item $|N_H(A)|=|N_H(B)|$;
  \item for any $C\in \mathcal{I}\setminus\{A,B\}$, $\eta(A,C)=\eta(B,C)$ and $\eta(C,A)=\eta(C,B)$.
 \end{enumerate} 
It is easy to see that $\sim$ is an equivalence relation on $\mathcal{I}$. 

\begin{lemma}
	With the notation above, if $A\sim B\in \cali$, we have $|\delta(A)|=|\delta(B)|$.
\end{lemma}

\begin{proof}
	Suppose that $\cali=\{C_1,C_2,\dots, C_n\}$, where $A=C_1$ and $B=C_2$, and set $M=[\eta(C_i,C_j)]_{n\times n}$. By the proof of Lemma~\ref{lemma existence of inversion formula on poset}, $M^{-1}=[\mu_{\eta}(C_i,C_j)]_{n\times n}$. Let 
	\[J=\begin{bmatrix}
		0&1&&&&\\
		1&0&&&&\\
		&&1&0&\dots&0\\
		&&0&1&\dots&0\\
		&&\vdots&\vdots&\ddots&\vdots\\
		&&0&0&\dots&1
	\end{bmatrix}.\]
	be a permutation matrix. As $A\sim B$, by the definition of $\sim$ we have $J^{-1}M J=M$ from a straightforward matrix calculation. It follows that $J^{-1}M^{-1} J=M^{-1}$. 
	By Lemma~\ref{lemma Delta}(ii), we have $|\Delta(C_1)|=|\Delta(A)|=|\Delta(B)|=|\Delta(C_2)|$, which implies 
	\begin{align*}
	\begin{bmatrix}
    |\Delta(C_1)|\\
    |\Delta(C_2)|\\
    \vdots\\
    |\Delta(C_n)|
    \end{bmatrix}
    &=
    \begin{bmatrix}
    |\Delta(C_2)|\\
    |\Delta(C_1)|\\
    \vdots\\
    |\Delta(C_n)|
    \end{bmatrix}
    =
    J
    \begin{bmatrix}
    |\Delta(C_1)|\\
    |\Delta(C_2)|\\
    \vdots\\
    |\Delta(C_n)|
    \end{bmatrix}.
	\end{align*}
	Thus,
	\begin{align*}
	\begin{bmatrix}
    |\delta(C_1)|\\
    |\delta(C_2)|\\
    \vdots\\
    |\delta(C_n)|
    \end{bmatrix}
    &=
    M^{-1}
    \begin{bmatrix}
    |\Delta(C_1)|\\
    |\Delta(C_2)|\\
    \vdots\\
    |\Delta(C_n)|
    \end{bmatrix}=
    J^{-1}M^{-1}J
    \begin{bmatrix}
    |\Delta(C_1)|\\
    |\Delta(C_2)|\\
    \vdots\\
    |\Delta(C_n)|
    \end{bmatrix}\\&=
    J^{-1}M^{-1}
    \begin{bmatrix}
    |\Delta(C_1)|\\
    |\Delta(C_2)|\\
    \vdots\\
    |\Delta(C_n)|
    \end{bmatrix}=
    J^{-1}
    \begin{bmatrix}
    |\delta(C_1)|\\
    |\delta(C_2)|\\
    \vdots\\
    |\delta(C_n)|
    \end{bmatrix}=
    \begin{bmatrix}
    |\delta(C_2)|\\
    |\delta(C_1)|\\
    \vdots\\
    |\delta(C_n)|
    \end{bmatrix}.
	\end{align*}
	This gives $|\delta(A)|=|\delta(C_1)|=|\delta(C_2)|=|\delta(B)|$.
	% and hence for each $C\in \cali\setminus\{A,B\}$, $\mu_\eta(C,A)=\mu_\eta(C,B)$ and $\mu_\eta(A,C)=\mu_\eta(B,C)$.
	% From Lemma~\ref{lemma Delta}(ii) and Lemma~\ref{lemma inversion formula on cali}(ii), we conclude that $|\delta(A)|=|\delta(B)|$.
\end{proof}

Suppose $\tilde{\mathcal{I}}$ is a set of equivalence class representatives of $\mathcal{I}$ with respect to $\sim$.  For any subgroup $A\in\mathcal{I}$, let $\tilde{A}=\{B\in\mathcal{I}\mid  B\sim A\}$ be the equivalence class containing $A$.

\begin{lemma}
\label{lemma m(G,H,n) tilde}
We have
$$m(G,H,n)=\frac{1}{n} \sum_{A\in \tilde{\cali}, |A|=\frac{|H|}{n}} \frac{|\delta(A)|\cdot|\tilde{A}|}{|N_H(A)|}.$$
\end{lemma}

\begin{proof}
    This is directly given by Lemma \ref{lemma m(G,H,n) first observation} and the definition of $\sim$.
\end{proof}

Here is a corresponding lemma to Lemma \ref{lemma inversion formula on cali}, which is the main technique we use to determine the valency as well as the multiplicities of subdegrees.

\begin{lemma}\label{lemma inversion formula on tilde cali}
	The following statements hold.
	\begin{enumerate}[{\rm (i)}]
		\item If $\tilde{\eta}$ is a function defined on $\tilde{\cali}\times\tilde{\cali}$ such that $$\tilde{\eta}(A,B)=\sum_{C\in \tilde{B}}\eta(A,C),$$
		then
		\begin{equation*}
			|\Delta(A)|=\sum_{B\in \tilde{\mathcal{I}}} \tilde{\eta}(A,B)|\delta(B)|
		\end{equation*}
		and
		\begin{equation*}
		\tilde{\eta}(A,B)=\begin{cases}
		1& 	\mbox{if $A=B$;}\\
		\eta(A,B)|\tilde{B}|& \mbox{otherwise.}\\
		\end{cases}
		\end{equation*}
		\item There exists a function $\mu_{\tilde{\eta}}:\tilde{\mathcal{I}}\times \tilde{\mathcal{I}}\rightarrow \mathbb{C}$, such that for each $B\in \tilde{\cali}$,
		\[ |\delta(B)|=\sum_{A\in \tilde{\mathcal{I}}} \mu_{\tilde{\eta}}(B,A)|\Delta(A)|.\]
	\end{enumerate}
\end{lemma}

\begin{proof}
	By (\ref{eqn eta}),
	  \begin{align*}
	  |\Delta(A)|
	  &=\sum_{B\in \cali}\eta(A,B)|\delta(B)| \\
	  &=\sum_{B\in \tilde{\cali}} \left(\sum_{C \in \tilde{B} }\eta(A,C)|\delta(C)|\right)\\
	  &=\sum_{B\in \tilde{\cali}} \left(\sum_{C \in \tilde{B} }\eta(A,C)\right)|\delta(B)|\\
	  &=\sum_{B\in \tilde{\cali}} \tilde{\eta}(A,B)|\delta(B)|.
	  \end{align*}
	  If $A=B$, then \[\tilde{\eta}(A,A)=\sum_{C \in \tilde{A} }\eta(A,C) =|\{C^h\mid C^h\ge A\ \mbox{and}\ C\sim A\}|=1.\]
	  If $A\ne B\in \tilde{\cali}$, then 
	  \[\tilde{\eta}(A,B)=\sum_{C \in \tilde{B} }\eta(A,C) =\sum_{C \in \tilde{B} }\eta(A,B) =\eta(A,B)|\tilde{B}|.\]
	  Thus, part~(i) holds. Now, by Lemma~\ref{inversion formula on poset}, part~(ii) holds.
\end{proof}

\begin{remark}\label{remark matrix form of the main method}
Suppose $\mathcal{I}$ is the set of subgroups $A$ of $H$ (up to conjugacy in $H$) such that $\delta(A)$ is non-empty. 
We can make the corresponding set $\tilde{\cali}$ into a partially ordered set $(\tilde{\cali},\preccurlyeq)$ by defining $A\preccurlyeq B$ if and only if $\tilde{\eta}(A,B)>0$. 
Then by Szpilrajn’s lemma, which asserts that every finite partial order is contained in a total order (see \cite[Lemma 1.2.1]{Rota}), $\tilde{\cali}$ can be written as $\{C_1,\dots,C_n\}$ such that $i\le j$ whenever $C_i\preccurlyeq C_j$.
% Note that $\tilde{\cali}$ is a subset of the partially ordered set $\cali$. This can be extened to a totally ordered set, say $(\tilde{\cali},\preccurlyeq)$. Write $\tilde{\cali}=\{C_1,\dots,C_n\}$ with $C_i\preccurlyeq C_j$ if and only if $i\le j$. 
Define an $n\times n$ matrix $M=[\tilde{\eta}(C_i,C_j)]_{n\times n}$.
% Note also that $\tilde{\eta}(C_i,C_j)\ne 0$ implies $C_i\preceq C_j$ in the partially ordered set $(\tilde{\cali},\preceq)$ and so $C_i\preccurlyeq C_j$ in the extended totally ordered set $(\tilde{\cali},\preccurlyeq)$. 
It follows that $M$ is an upper-triangular matrix with diagonal entries all being $1$. Now Lemma \ref{inversion formula on poset} applies, making $M^{-1}=[\mu_{\tilde{\eta}}(C_i,C_j)]_{n\times n}$ and $\mu_{\tilde{\eta}}$ is unique. In the language of matrices, the above equations are exactly
\begin{equation*}
\begin{aligned}
	\Delta:=
    \begin{bmatrix}
    |\Delta(C_1)|\\
    |\Delta(C_2)|\\
    \vdots\\
    |\Delta(C_n)|
    \end{bmatrix}
    &=
    M
    \begin{bmatrix}
    |\delta(C_1)|\\
    |\delta(C_2)|\\
    \vdots\\
    |\delta(C_n)|
    \end{bmatrix}\\
    &=
    \begin{bmatrix}
    \tilde{\eta}(C_1,C_1)&\tilde{\eta}(C_1,C_2)&\cdots&\tilde{\eta}(C_1,C_n)\\
    \tilde{\eta}(C_2,C_1)&\tilde{\eta}(C_2,C_2)&\cdots&\tilde{\eta}(C_2,C_n)\\
    \vdots&\vdots&\ddots&\vdots\\
    \tilde{\eta}(C_n,C_1)&\tilde{\eta}(C_n,C_2)&\cdots&\tilde{\eta}(C_n,C_n)
    \end{bmatrix}
    \begin{bmatrix}
    |\delta(C_1)|\\
    |\delta(C_2)|\\
    \vdots\\
    |\delta(C_n)|
    \end{bmatrix}
    \end{aligned}
\end{equation*}
and
\begin{equation*}
\begin{aligned}
	\delta:=
    \begin{bmatrix}
    |\delta(C_1)|\\
    |\delta(C_2)|\\
    \vdots\\
    |\delta(C_n)|
    \end{bmatrix}
    &=
   M^{-1}
    \begin{bmatrix}
    |\Delta(C_1)|\\
    |\Delta(C_2)|\\
    \vdots\\
    |\Delta(C_n)|
    \end{bmatrix}\\
    &=
     \begin{bmatrix}
    \mu_{\tilde{\eta}}(C_1,C_1)&\mu_{\tilde{\eta}}(C_1,C_2)&\cdots&\mu_{\tilde{\eta}}(C_1,C_n)\\
    \mu_{\tilde{\eta}}(C_2,C_1)&\mu_{\tilde{\eta}}(C_2,C_2)&\cdots&\mu_{\tilde{\eta}}(C_2,C_n)\\
    \vdots&\vdots&\ddots&\vdots\\
    \mu_{\tilde{\eta}}(C_n,C_1)&\mu_{\tilde{\eta}}(C_n,C_2)&\cdots&\mu_{\tilde{\eta}}(C_n,C_n)
    \end{bmatrix}
    \begin{bmatrix}
    |\Delta(C_1)|\\
    |\Delta(C_2)|\\
    \vdots\\
    |\Delta(C_n)|
    \end{bmatrix},
    \end{aligned}
\end{equation*}
where both $M$ and $M^{-1}$ are upper-triangular with diagonal entries all equal to $1$. Therefore, to calculate $|\delta(C_1)|=|\delta(1)|=|H|\cdot \val(G,H)$ and the multiplicity of each subdegree, it suffices to determine the matrix $M$ and all the values $|\Delta(C_i)|$. We will adopt the notation introduced in this remark later in the text.
\end{remark}

\begin{example}
Let $G=\PGL_3(7)$ be a primitive group with stabiliser $H=\langle x\rangle\times(\langle y\rangle{:}\langle\sigma\rangle)\cong\mathbb{Z}_3\times(\mathbb{Z}_{19}{:}\mathbb{Z}_3)$. In this case, the only possible arc stabilisers up to conjugacy in $H$ are $$\cali=\{1,\langle \sigma\rangle,\langle x\sigma\rangle,\langle x^2\sigma\rangle,\langle x,\sigma\rangle,H\},$$
in which $\langle x\sigma\rangle$ and $\langle x^2\sigma\rangle$ are conjugate in $G$. In particular, it is not hard to see that $\langle x\sigma\rangle\sim\langle x^2\sigma\rangle$ by checking conditions (E1)-(E3) and so we have
$$\tilde{\cali}=\{1,\langle\sigma\rangle,\langle x\sigma\rangle,\langle x,\sigma\rangle,H\}.$$
The matrix $M$ in Remark \ref{remark matrix form of the main method} in this case is
\begin{equation*}
    M=
    \begin{bmatrix}
    1&19&38&19&1\\
    0&1&0&1&1\\
    0&0&1&1&1\\
    0&0&0&1&1\\
    0&0&0&0&1
    \end{bmatrix}.
\end{equation*}
Moreover, we have $\mathcal{S}\cap \langle x\sigma\rangle^G=\{\langle x\rangle,\langle x\sigma\rangle,\langle x^2\sigma\rangle\}$, while $\mathcal{S}\cap A^G=\{A\}$ for any other elements $A$ in $\tilde{\cali}$ here. This gives, by applying Lemma \ref{lemma Delta}, that
\begin{equation*}
    \Delta=
    \begin{bmatrix}
    |\Delta(1)|\\
    |\Delta(\langle\sigma\rangle)|\\
    |\Delta(\langle x\sigma\rangle)|\\
    |\Delta(\langle x,\sigma\rangle)|\\
    |\Delta(H)|
    \end{bmatrix}
    =
    \begin{bmatrix}
    5630688\\
    4104\\
    6669\\
    513\\
    171
    \end{bmatrix}.
\end{equation*}
For more details about computing the order of normalisers in a more general setting, one can refer to Corollary \ref{corollary normaliser r^2} and Lemma \ref{lemma N} (see Section \ref{section odd valencies}). Finally, we have
\begin{equation*}
    \delta=M^{-1}\Delta
    =
    \begin{bmatrix}
    5321862\\
    3591\\
    6156\\
    342\\
    171
    \end{bmatrix}.
\end{equation*}
It follows that $\val(G,H)=5321862/171=31122$ by Lemma \ref{lemma delta}. Other multiplicities of subdegrees can be also obtained by applying Lemma \ref{lemma m(G,H,n) tilde}. We refer the reader to Proposition \ref{proposition val PGL C3 n prime} for a general statement.
\end{example}

\section{Proof of Theorem \ref{thm Frobenius val intro}}

\label{section Frobenius}

In this section, we assume $G$ is a primitive permutation group with stabiliser $H$, and $H=K{:}L$ is a Frobenius group with a cyclic Frobenius kernel $K$. We aim to determine $\tilde{\mathcal{I}}$ and $\val(G,H)$ in this situation.

As $G$ is a primitive permutation group with stabiliser $H$, it is well known that $H$ is maximal and core-free in $G$. The next lemma records some further properties.

\begin{lemma}
    \label{lemma Frobenius plus}
    The following statements hold.
    \begin{enumerate}[{\indent\rm (i)}]
        \item For any non-trivial subgroup $A$ of $K$, $N_G(A)=H$.
        \item For any $g\in G$, if $H\cap H^g\le K$ then $H\cap H^g=1$.
        \item For any proper subgroup $M$ of $H$ with $M\cap K\ne 1$, $\{g\in G\mid M\le H^g\}=H.$
    \end{enumerate}
\end{lemma}

\begin{proof}
    \begin{enumerate}[{\indent\rm (i)}]
        \item Firstly $A$ is a characteristic subgroup of $K$ and $K\lhd H$. This implies $A\lhd H$ and so $N_G(A)\ge H$. On the other hand, $A$ is not normal in $G$ because $H$ is core-free. It follows that $N_G(A)=H$ since $H$ is maximal in $G$.
        \item If $1\ne H\cap H^g\le K$ then $g\in N_G(H\cap H^g)=H$ because $H\cap H^g$ is the only subgroup in $H$ with this order. This makes $H\cap H^g=H$, a contradiction.
        \item It is straightforward to see that every $g\in H$ makes $M\le H^g=H$. If $g\notin H$ satisfies the above, then $1\ne M\cap K\le M\le H\cap H^g$, a contradiction to (ii).
    \end{enumerate}
\end{proof}

Suppose that $A$ is a proper non-trivial subgroup of $H$. If $A\cap K\neq 1$, then by Lemma~\ref{lemma Frobenius plus}(iii),  $\delta_H(A)=\emptyset$. If $A\cap K=1$, then by Lemma~\ref{lemma Frobenius}(iii), there exists $k\in K$ such that $A^k\le L$.  Hence, without loss of generality, we set
the representatives in $\tilde{\mathcal{I}}$ to be conjugates of subgroups of $L$ together with $H$.  Thus,
\begin{equation*}
    \val(G,H)
    =|G:H|-1-\frac{1}{|L|}\sum_{1\ne A\le L}|\delta_H(A)|.
\end{equation*}
It suffices to calculate $|\delta_H(A)|$ for $1\ne A\le L$, and we can reduce our calculation to the conjugates in $L$.

\begin{lemma}
    If $1<A<L$, then $$|\delta_H(A)|=|K|\cdot |\delta_L(A)|,$$while $$|\delta_H(L)|=|K|\cdot(|\delta_L(L)|-|L|).$$
\end{lemma}

\begin{proof}
    For any $1<A\le L$, by Lemma~\ref{lemma Frobenius}(iii), 
\begin{equation}
	\label{eqn6}
	\begin{aligned}
	 |\delta_H(A)|	&=\left| \bigcup_{k\in K}\{g\in G\mid H\cap H^g=A, A^{g^{-1}}\leqslant L^k\}\right|\\
						&=\sum_{k\in K}| \{g\in G\mid H\cap H^g=A, A^{g^{-1}}\leqslant L^k\}|\\
						&=\sum_{k\in K}| k\{g\in G\mid H\cap H^g=A, A^{g^{-1}}\leqslant L\}|\\
						&=|K|\cdot| \{g\in G\mid H\cap H^g=A, A^{g^{-1}}\leqslant L\}|.\\
	\end{aligned}
\end{equation}

Suppose that $1<A<L$. For any $x\in \delta_L(A)$, it is easy to see that $x\notin H$ and $A^{x^{-1}}<L$. Set $B=H\cap H^x\geqslant L\cap L^x=A>1$. Since $B\in \tilde{\mathcal{I}}$, $B\cap K=1$.
By Lemma~\ref{lemma Frobenius}(iii), there exists a unique $k\in K$ such that $A^k \le B^k\le L$. Note that $A \le L$. By the uniqueness of $k$, we have $k=1$ and $B\le L$. Similarly, we have $A^{x^{-1}}\le B^{x^{-1}}\le L$. Therefore $B\le L\cap L^x=A$, and $x\in \{g\in G\mid H\cap H^g=A, A^{g^{-1}}\le L\}$. Conversely, for any $y\in \{g\in G\mid H\cap H^g=A, A^{g^{-1}}\le L\}$, we have $A\le L^y$. Moreover, $A\le L\cap L^y\le H\cap H^y=A$. This implies that $y\in \delta_L(A)$. Thus, $\delta_L(A)=\{g\in G\mid H\cap H^g=A, A^{g^{-1}}\le L\}$. Applying this equality to (\ref{eqn6}) we conclude that
\begin{equation*}
	% \label{eqn7}
	\begin{aligned}
	 |\delta_H(A)|	&=|K|\cdot| \{g\in G\mid H\cap H^g=A, A^{g^{-1}}\leqslant L\}|=|K|\cdot |\delta_L(A)|.
	\end{aligned}
\end{equation*}

Now, suppose that $A=L$. Then
\begin{equation*}
	\begin{aligned}
		\{g\in G\mid H\cap H^g=L, L^{g^{-1}}\leqslant L\}	
		=& \{g\in G\mid H\cap H^g=L\}\cap \{g\in G\mid L^{g^{-1}}\leqslant L\}\\
		=& N_G(L)\cap \{g\in G\mid H\cap H^g=L\}\\
		% =& \{g\in N_G(L)\mid H\cap H^g=L\}\\
		=& N_G(L)\setminus H\\=&N_G(L)\setminus L\\
		=& \delta_L(L)\setminus L.
	\end{aligned}
\end{equation*}
Applying this equality to (\ref{eqn6}) we get 
\begin{equation*}
|\delta_H(L)|	=|K|\cdot |\delta_L(L)\setminus L|=|K|\cdot(|\delta_L(L)|-|L|).\end{equation*}
This completes the proof.
\end{proof}

Write $L=\langle y\rangle$. 
Let $d$ be a proper divisor of $|L|$.
For any  $x\in N_G(\langle y^d\rangle)$, it is obvious that $x\in \Delta_L(\langle y^d\rangle)$. Conversely, for any $x\in \Delta_L(\langle y^d\rangle)$, notice that $\langle y^d\rangle $ is the unique subgroup of index $d$ both in $L$ and $L^x$, $x\in N_G(\langle y^d\rangle)$. It follows that
\begin{equation*}
|N_G(\langle y^d\rangle)|=|\Delta_L(\langle y^d\rangle)|=\sum_{e\mid d}|\delta_{L}(\langle y^e\rangle)|.
% \label{equation intersection PIE}
\end{equation*}
It follows by the M\"{o}bius inversion formula that
\begin{equation*}
|\delta_{L}(\langle y^d\rangle)|=\sum_{e\mid d}\mu(e)|N_G(\langle y^{\frac{d}{e}}\rangle)|.
\end{equation*}
Indeed, the function $\mu_{\tilde{\eta}}$ in Lemma \ref{lemma inversion formula on tilde cali} is exactly the M\"obius function on integers in this case.

Considering the number of suborbits of length $d|K|$, it is easy to see that
\begin{equation*}
|H|\cdot d|K| \cdot m\left(G,H,d|K|\right)=\left|\left\{g\in G\mid |H\cap H^g|=\frac{|L|}{d}\right\}\right|.
\end{equation*}
By Lemma~\ref{lemma Frobenius}(iii), 
\begin{equation}\label{equation m(G,H,n) Frobenius}
\begin{aligned}
m\left(G,H,d|K|\right)
&=\frac{1}{d|H|\cdot|K|}\left|\left\{g\in G\mid |H\cap H^g|=\frac{|L|}{d}\right\}\right|\\
&=\frac{1}{d|H|}\left|\{g\in G\mid H\cap H^g=\langle y^d\rangle\}\right|\\
&=\frac{1}{d|H|}\left|\delta_H(\langle y^d\rangle)\right|\\
&=\begin{cases}
\frac{1}{d|L|}\left|\delta_L(\langle y^d\rangle)\right| & 1<d<|L|\\
\frac{1}{|L|}(|\delta_L(L)|-|L|) & d=1
\end{cases}\\
&=\begin{cases}
\frac{1}{d|L|}\sum_{e\mid d}\mu(e)|N_G(\langle y^{\frac{d}{e}}\rangle)| & 1<d<|L|\\
\frac{1}{|L|}(|N_G(L)|-|L|) & d=1.
\end{cases}\\
\end{aligned}
\end{equation}

Now, we are ready to compute $\val(G,H)$. Firstly,
\begin{equation}\label{equation val Frobenius 1}
\begin{aligned}
	\val(G,H)
	&=|G:H|-1-\frac{1}{|L|}\sum_{1< A\le L}|\delta_H(A)|\\
	&=|G:H|-1-\frac{|K|}{|L|}\cdot (|\delta_L(L)|-|L|)-\frac{|K|}{|L|}\sum_{1< A< L}|\delta_L(A)|\\
	&=|G:H|-1-\frac{|K|}{|L|}\cdot \left(-|L|+\sum_{1< A\le L}|\delta_L(A)|\right)\\
 	&=|G:H|+|K|-1-\frac{|K|}{|L|}\sum_{1< A\le L}|\delta_L(A)|.\\
% 	&=|G:H|+|K|-1-\frac{|K|}{|L|}\sum_{d\mid |L|, d\ne |L|}|\delta_L(\langle y^d\rangle)|\\
% 	&=|G:H|+|K|-1-\frac{|K|}{|L|}\sum_{d\mid |L|, d\ne |L|}\sum_{e\mid d}\mu(e)|N_G(\langle y^{\frac{d}{e}}\rangle)|\\
% 	&=|G:H|+|K|-1-\frac{|K|}{|L|}\sum_{e\mid |L|, e\ne |L|}\left(\sum_{\frac{d}{e}\mid \frac{|L|}{e}, \frac{d}{e}\ne\frac{|L|}{e}}\mu(\frac{d}{e})\right)|N_G(\langle y^{e}\rangle)|\\
% 	&=|G:H|+|K|-1+\frac{|K|}{|L|}\sum_{e\mid |L|, e\ne |L|} \mu(\frac{|L|}{e}) |N_G(\langle y^{e}\rangle)|\\
%	&=|G:H|+|K|-1+\frac{|K|}{|L|}\sum_{1\ne d\mid |L|} \mu(d) |N_G(\langle y^{\frac{|L|}{d}}\rangle)|.\\
\end{aligned}
\end{equation}
Moreover,
\begin{equation}\label{equation val Frobenius 2}
    \begin{aligned}
        \sum_{1<A\le L}|\delta_L(A)|&=\sum_{d\mid |L|, d\ne |L|}|\delta_L(\langle y^d\rangle)|\\
        &=\sum_{d\mid |L|, d\ne |L|}\sum_{e\mid d}\mu(e)|N_G(\langle y^{\frac{d}{e}}\rangle)|\\
        &=\sum_{e\mid |L|, e\ne |L|}\left(\sum_{\frac{d}{e}\mid \frac{|L|}{e}, \frac{d}{e}\ne\frac{|L|}{e}}\mu\left(\frac{d}{e}\right)\right)|N_G(\langle y^{e}\rangle)|\\
        &=-\sum_{e\mid |L|, e\ne |L|} \mu\left(\frac{|L|}{e}\right) |N_G(\langle y^{e}\rangle)|\\
        &=-\sum_{1\ne e\mid |L|} \mu(e) |N_G(\langle y^{\frac{|L|}{e}}\rangle)|.
    \end{aligned}
\end{equation}

To summarise (\ref{equation m(G,H,n) Frobenius}), (\ref{equation val Frobenius 1}) and (\ref{equation val Frobenius 2}) we have the following theorem.

\begin{thm}\label{thm Frobenius val}
    Suppose $G$ is a primitive permutation group with stabiliser $H$, where $H=K{:}L$ is Frobenius with cyclic kernel $K$. Then
    \begin{equation*}
    \val(G,H)=|G:H|+|K|-1+\frac{|K|}{|L|}\sum_{1\ne d\mid |L|} \mu(d) |N_G(\langle y^{\frac{|L|}{d}}\rangle)|,
\end{equation*}
where $L=\langle y\rangle$ and $\mu$ is the M\"{o}bius function. Moreover, all the other non-trivial subdegrees are $d|K|$ for proper divisors $d$ of $|L|$, with multiplicities
\begin{equation*}
m\left(G,H,d|K|\right)=
\begin{cases}
\frac{1}{d|L|}\sum_{e\mid d}\mu(e)|N_G(\langle y^{\frac{d}{e}}\rangle)| & 1<d<|L|\\
\frac{1}{|L|}(|N_G(L)|-|L|) & d=1.
\end{cases}
\end{equation*}
\end{thm}

\section{Alternating and Symmetric Groups}

\label{section An and Sn}

Let $G$ be an almost simple primitive group with stabiliser $H$. If we have $b(G)=2$, $\soc(G)=A_n$ and $H$ is soluble, then by \cite{Burness2020base,SolubleStabiliser}, $(G,H)$ is exactly one of the pairs given in Table \ref{table val almost simple soluble stabiliser alternating}. Our aim is to calculate $\val(G,H)$ for each pair listed in Table \ref{table val almost simple soluble stabiliser alternating}.

There are only two infinite families in Table \ref{table val almost simple soluble stabiliser alternating}, both of which have a Frobenius stabiliser with cyclic kernel. Hence, Theorem \ref{thm Frobenius val} can be directly applied. Other valencies in Table \ref{table val almost simple soluble stabiliser alternating} can be easily calculated by {\sc Magma}.

\subsection{\texorpdfstring{$G=S_p$ and $H=\AGL_1(p)$}{G=Sp and H=AGL1(p)}}

In this case, $K\cong \mathbb{Z}_p$ and $L=\langle y\rangle\cong \mathbb{Z}_{p-1}$ with $y$ a $(p-1)$-cycle in $S_p$. Then $y^{\frac{|L|}{d}}$ is a product of $\frac{p-1}{d}$ disjoint $d$-cycles. We have
\begin{equation*}
    |N_G(\langle y^{\frac{|L|}{d}}\rangle)|=\phi(d)|C_G(\langle y^{\frac{|L|}{d}}\rangle)|=\phi(d)|\mathbb{Z}_d\wr S_{\frac{p-1}{d}}|=\phi(d)d^{\frac{p-1}{d}}\left(\frac{p-1}{d}\right)!.
\end{equation*}
Therefore, by Theorem \ref{thm Frobenius val}, we have
\begin{equation*}
    \begin{aligned}
    \val(S_p,\AGL_1(p))&=\val(G,H)\\&=|G:H|+|K|-1+\frac{|K|}{|L|}\sum_{1\ne d\mid |L|}\mu(d)|N_G(\langle y^{\frac{|L|}{d}}\rangle)|\\
%    &=(p-2)!+p-1+\frac{p}{p-1}\sum_{1\ne d\mid (p-1)}\mu(d)\phi(d)d^{\frac{p-1}{d}}\left(\frac{p-1}{d}\right)!\\
    &=(p-2)!+p-1+p\sum_{1\ne d\mid (p-1)}\mu(d)\phi(d)d^{\frac{p-1}{d}-1}\left(\frac{p-1}{d}-1\right)!,
    \end{aligned}
\end{equation*}
which is exactly (\ref{equation val Sp}), and also
\begin{equation*}
    m(G,H,dp)=
    \begin{cases}
    \frac{1}{d(p-1)}\sum_{e\mid d}\mu(e)\phi(\frac{(p-1)e}{d})\cdot(\frac{(p-1)e}{d})^{\frac{d}{e}}\cdot(\frac{d}{e})! & 1<d<p-1\\
    \phi(p-1)-1 & d=1
    \end{cases}
\end{equation*}
for proper divisors $d$ of $p-1$.

\subsection{\texorpdfstring{$G=A_p$ and $H=\AGL_1(p)\cap A_p$}{G=Ap and H=AGL1(p) cap Ap}}

In this case, $K\cong\mathbb{Z}_p$ and $L=\langle y\rangle\cong\mathbb{Z}_{\frac{p-1}{2}}$ with $y$ a product of two disjoint $\frac{p-1}{2}$-cycles. Then $y^{\frac{|L|}{d}}$ is a product of $\frac{p-1}{d}$ disjoint $d$-cycles. We have

\begin{equation*}
        |N_G(\langle y^{\frac{|L|}{d}}\rangle )|=\phi(d)|C_G(\langle y^{\frac{|L|}{d}}\rangle)|=\phi(d)|\mathbb{Z}_d\wr S_{\frac{p-1}{d}}\cap A_p|=\frac{1}{2}\phi(d)d^{\frac{p-1}{d}}\left(\frac{p-1}{d}\right)!.
\end{equation*}
Again by Theorem \ref{thm Frobenius val}, we have
\begin{equation*}
    \begin{aligned}
    \val(A_p,\AGL_1(p)\cap A_p)&=\val(G,H)\\&=|G:H|+|K|-1+\frac{|K|}{|L|}\sum_{1\ne d\mid |L|}\mu(d)|N_G(\langle y^{\frac{|L|}{d}}\rangle)|\\
%    &=(p-2)!+p-1+\frac{2p}{p-1}\sum_{1\ne d|\frac{p-1}{2}}\mu(d)\frac{1}{2}\phi(d)d^{\frac{p-1}{d}}\left(\frac{p-1}{d}\right)!\\
    &=(p-2)!+p-1+p\sum_{1\ne d\mid\frac{p-1}{2}}\mu(d)\phi(d)d^{\frac{p-1}{d}-1}\left(\frac{p-1}{d}-1\right)!,
    \end{aligned}
\end{equation*}
which meets (\ref{equation val Ap}), and also
\begin{equation*}
    m(G,H,dp)=
    \begin{cases}
    \frac{1}{d(p-1)}\sum_{e\mid d}\mu(e)\phi(\frac{(p-1)e}{2d})\cdot (\frac{(p-1)e}{2d})^{\frac{2d}{e}}\cdot(\frac{2d}{e})! & 1<d<\frac{p-1}{2}\\
    \phi(\frac{p-1}{2})\cdot \frac{p-1}{2}-1 & d=1
    \end{cases}
\end{equation*}
for proper divisors $d$ of $\frac{p-1}{2}$.

\section{Prime-power Valencies}
\label{section prime power}

Let $G$ be an almost simple primitive group with stabiliser $H$. If $\val(G,H)$ is a prime-power, then so is $|H|$ and hence $H$ is soluble. The possibilities for $(G,H)$ given in \cite{SolubleStabiliser} are presented in Table \ref{table almost simple stabiliser 2-group pre}. In the table we use the same notation as in \cite{SolubleStabiliser}, where $G_0\lhd G$ is minimal such that $H_0:=H\cap G_0$ is maximal in $G_0$ and $H=H_0.(G/G_0)$.

\begin{table}[ht]
	\centering
	\begin{tabular}{lll}
		\Xhline{2pt}
		$G_0$&$H_0$&Conditions\\
		\Xhline{1pt}
		$\LL_2(q)$&$D_{2(q-1)/(2,q-1)}$&$q\ne 5,7,9,11$\\
		&$D_{2(q+1)/(2,q-1)}$&$q\ne 7,9$\\
		$\PGL_2(7)$&$D_{16}$&\\
		$\PGL_2(9)$&$D_{16}$&\\
		$\LL_2(9).2\cong \M_{10}$&$8{:}2$&\\
		$\LL_r^\epsilon(q)$&$\mathbb{Z}_a{:}\mathbb{Z}_r$&$r\ge 3$ prime, $a=\frac{q^r-\epsilon}{(q-\epsilon)(r,q-\epsilon)}$, $G_0\ne\UU_3(3),\UU_5(2)$\\
		\Xhline{2pt}
	\end{tabular}
	\caption{Possible cases of almost simple primitive group $G$ with a maximal subgroup $H$ of prime-power order}
	\label{table almost simple stabiliser 2-group pre}
	
\end{table}

Easy applications of Catalan's conjecture (now a theorem proved in \cite{CatalanConjecture}) and Zsigmondy's theorem (see \cite{Zsigmondy}) allow us to eliminate some possibilities in Table \ref{table almost simple stabiliser 2-group pre}. We record the two theorems below.

\begin{thm}[Catalan's conjecture]\label{thm Catalan Conjecture}
    The only solution in the natural numbers to
    \begin{equation*}
        x^a-y^b=1
    \end{equation*}
    for $a,b>1$, $x,y>0$ is $(a,b,x,y)=(2,3,3,2)$.
\end{thm}

\begin{thm}[Zsigmondy]\label{thm Zsigmondy}
    Let $n>1$, $a>1$ and $b$ be positive integers such that $(a,b)=1$, and $\epsilon=\pm 1$. There exists a prime divisor $p$ of an $a^n-\epsilon b^n$ such that $p$ does not divide $a^j- \epsilon b^j$ for all $j$ with $0 < j < n$,
except exactly in the following cases:
\begin{enumerate}[\rm (i)]
    \item $\epsilon=1$, $n=2$, $a+b=2^s$ for some $s\ge 2$;
    \item $(\epsilon,n,a,b)=(1,6,2,1)$;
    \item $(\epsilon,n,a,b)=(-1,3,2,1)$.
\end{enumerate}
\end{thm}

In the following proposition, we eliminate some of the possibilities in Table \ref{table almost simple stabiliser 2-group pre} and we classify the almost simple primitive groups with point stabilisers of prime-power order.

\begin{proposition}\label{proposition almost simple stabiliser p-group}
	Let $G$ be an almost simple primitive group with stabiliser $H$. If $|H|$ is a prime power, then $H$ is a $2$-group and $(G,H)$ is listed in Table \ref{table almost simple stabiliser 2-group}.
\end{proposition}

\begin{table}[ht]
	\centering
	\begin{tabular}{lll}
		\Xhline{2pt}
		$G$&$H$&Conditions\\
		\Xhline{1pt}
		$\LL_2(p)$&$D_{p-1}$&$p\ge 17$ is a Fermat prime\\
		&$D_{p+1}$&$p\ge 31$ is a Mersenne prime\\
		$\PGL_2(p)$&$D_{2(p-1)}$&$p\ge 17$ is a Fermat prime\\
		&$D_{2(p+1)}$&$p\ge 7$ is a Mersenne prime\\
		$\PGL_2(9)$&$D_{16}$&\\
		$\M_{10}$&$8{:}2$&\\
		$\PGammaL_2(9)$&$8{:}2^2$&\\
		\Xhline{2pt}
	\end{tabular}
	\caption{Almost simple groups $G$ with a maximal subgroup $H$ of prime-power order}
	\label{table almost simple stabiliser 2-group}
	
\end{table}

\begin{proof}
	First observe that if $|H|$ is a prime power then so is $|H_0|$. In the first and the second cases in Table \ref{table almost simple stabiliser 2-group pre}, we need $q-1$ and $q+1$ to be a prime power, respectively, and thus a power of $2$ since $H_0$ is dihedral. Hence, $q$ is odd, and then Theorem \ref{thm Catalan Conjecture} can be applied so that $q$ must be a prime of the form $q = 2^f\pm 1$ for some $f$ (note that $q\ne 9$). Thus, $q$ is a Fermat and a Mersenne prime, respectively. This (together with the case when $(G,H)=(\PGL_2(7),D_{16})$) gives the first four rows in Table \ref{table almost simple stabiliser 2-group}.
	
	In the last case in Table \ref{table almost simple stabiliser 2-group pre}, both $a$ and $r$ are odd. By Theorem \ref{thm Zsigmondy}, there exists a prime divisor $s$ of $q^r-\epsilon$ such that $s$ does not divide $q-\epsilon$, except when $(r,q,\epsilon)=(3,2,-1)$, in which case $G$ is not almost simple and we do not need to consider it. It follows that $s$ divides $ar$, and hence $s=r$ because we need $ar$ to be a prime power. However, $q\equiv q^r\equiv \epsilon\pmod r$ gives that $s=r$ divides $q-\epsilon$, leading to a contradiction. Therefore, the last row in Table \ref{table almost simple stabiliser 2-group pre} does not arise.
\end{proof}

Now we calculate the possible cases one by one to prove Theorem \ref{thm classification prime-power}.

\begin{proposition}
	Let $G = \PSL_2(q)$ with $q\ge 13$ odd, and $H\cong D_{q-1}$ a maximal subgroup of $G$. Then
	\begin{equation*}
	\val(G,H)=
	\begin{cases}
	\frac{1}{4}(q-1)(q+7)&q\equiv 1\Mod 4\\
	\frac{1}{4}(q-1)(q+5)&q\equiv 3\Mod 4.
	\end{cases}
	\end{equation*}
	\label{prop PSL C2}
\end{proposition}

\begin{proof}
    This is given in the proof of \cite[Lemma 4.7]{Burness2020base}. Here we give another proof using the strategy introduced in Section \ref{section Strategy}.
    
    Note that if $q\equiv 3\Mod 4$ then Theorem \ref{thm Frobenius val intro} can be directly applied. Thus, we only need to consider the case when $q\equiv 1\pmod 4$. Observe that the only possible cases for $H\cap H^g$ up to isomorphism are $1$, $\mathbb{Z}_2$, $\mathbb{Z}_2^2$ or $H$, and all involutions in $G$ are conjugate. Write $H=\langle x\rangle {:}\langle y\rangle$.
    
    First suppose $q\equiv 5\Mod 8$. In this case all subgroups of $H$ that are isomorphic to $\mathbb{Z}_2^2$ are conjugate in $H$. This implies $\tilde{\mathcal{I}}$ can be chosen as
    \begin{equation*}
    \tilde{\cali}=\{1,\langle y\rangle,\langle x^{\frac{q-1}{4}},y\rangle,H\},
    \end{equation*}
    which gives the matrix
    \begin{equation*}
    M=
    \begin{bmatrix}
    1&\frac{q-1}{2}&\frac{q-1}{4}&1\\
    &1&1&1\\
    &&1&1\\
    &&&1
    \end{bmatrix},
    \end{equation*}
    where the missing entries are zero (and similarly for the matrices presented below). To obtain the values $\Delta(A)$ for $A\in\tilde{\cali}$ by Lemma \ref{lemma Delta}, it suffices to find $N_G(\langle x^{\frac{q-1}{4}},y\rangle)$. Other normalisers can be easily determined. Indeed, $N_G(\langle x^{\frac{q-1}{4}},y\rangle)\cong A_4$. This gives
    \begin{equation*}
    \Delta=
    \begin{bmatrix}
    |G|\\
    (q-1)^2(\frac{1}{q-1}+\frac{1}{2})\\
    3(q-1)\\
    q-1
    \end{bmatrix}
    \end{equation*}
    and so $\val(G,H)$ follows.
    
    The case $q\equiv 1\Mod 8$ is slightly different since there are more elements in $\tilde{\cali}$. In this case the subgroups $\mathbb{Z}_2^2$ give two conjugacy classes in $H$, and they are not conjugate in $G$. It follows that
    \begin{equation*}
    \tilde{\cali}=\{1,\langle y\rangle,\langle xy\rangle,\langle x^{\frac{q-1}{4}},y\rangle,\langle x^{\frac{q-1}{4}},xy\rangle,H\}
    \end{equation*}
    is a choice of $\tilde{\cali}$, with the matrix
    \begin{equation*}
    M=
    \begin{bmatrix}
    1&\frac{q-1}{4}&\frac{q-1}{4}&\frac{q-1}{8}&\frac{q-1}{8}&1\\
    &1&0&1&0&1\\
    &&1&0&1&1\\
    &&&1&0&1\\
    &&&&1&1\\
    &&&&&1
    \end{bmatrix}.
    \end{equation*}
    Note also that in this case $N_G(\langle x^{\frac{q-1}{4}},y\rangle)\cong N_G(\langle x^{\frac{q-1}{4}},xy\rangle)\cong S_4$ and $N_H(\langle x^{\frac{q-1}{4}},y\rangle)\cong N_H(\langle x^{\frac{q-1}{4}},xy\rangle)\cong D_8$. The vector $\Delta$ is
    \begin{equation*}
    \Delta=
    \begin{bmatrix}
    |G|\\
    (q-1)^2(\frac{1}{q-1}+\frac{1}{2})\\
    (q-1)^2(\frac{1}{q-1}+\frac{1}{2})\\
    3(q-1)\\
    3(q-1)\\
    q-1
    \end{bmatrix},
    \end{equation*}
    which gives $\val(G,H)$.
\end{proof}

\begin{corollary}
	Let $p=2^f+1\ge 17$ be a Fermat prime, and let $G=\PSL_2(p)$ and $H=D_{p-1}$. Then $\val(G,H)=2^f(2^{f-2}-2)$, which is not a prime power.
	\label{corollary C2 PSL}
\end{corollary}

\begin{proposition}
	Let $G = \PGL_2(q)$ with $q\ge 7$, and $H\cong D_{2(q-1)}$ a maximal subgroup of $G$. Then the Saxl graph of $G$ with stabiliser $H$ is isomorphic to the Johnson graph $J(q+1,2)$.
	\label{prop PGL C2}
\end{proposition}

\begin{proof}
	We may identify $\Omega$ with the set of unordered pairs $\{U,W\}$ of $1$-dimensional subspaces of the natural module $V\cong \mathbb{F}_q^2$ with $U\oplus W=V$. It is easy to see that two unordered pairs of $1$-dimensional subspaces form a base if and only if they have exactly one $1$-dimensional subspace in common. See also \cite[Example 2.5]{TimSaxlGraph} or \cite[Table 2]{PGLsubdegree}.
\end{proof}

\begin{corollary}
	Let $p=2^f+1\ge 17$ be a Fermat prime, and let $G=\PGL_2(p)$ and $H=D_{2(p-1)}$. Then $\val(G,H)=2^{f+1}$, which is a prime power.
	\label{corollary C2 PGL}
\end{corollary}

\begin{proposition}
	\label{prop PSL C3}
	Let $G = \PSL_2(q)$ with $q\ge 11$ odd, and $H\cong D_{q+1}$ a maximal subgroup of $G$. Then
	\begin{equation*}
	\val(G,H)=
	\begin{cases}
	\frac{1}{4}(q+1)(q-1)&q\equiv 1\Mod 4\\
	\frac{1}{4}(q+1)(q-3)&q\equiv 3\Mod 4.
	\end{cases}
	\end{equation*}
\end{proposition}

\begin{proof}
    This follows from the proof of \cite[Lemma 7.9]{BurnessHarper2018}. We note that an argument similar to the proof of Proposition \ref{prop PSL C2} can also be applied.
\end{proof}

\begin{corollary}
	Let $p=2^f-1\ge 31$ be a Mersenne prime, $G=\PSL_2(p)$ and $H=D_{p+1}$. Then $\val(G,H)=2^{2f-2}-2^f$, which is not a prime power.
	\label{corollary C3 PSL}
\end{corollary}

\begin{proposition}
	Let $G=\PGL_2(q)$ and $H\cong D_{2(q+1)}$ a maximal subgroup of $G$. Then $\val(G,H)=0$.
\end{proposition}

\begin{proof}
	The subdegrees of $G$ are given in \cite[Table 2]{PGLsubdegree} and there is no regular suborbit. Hence, the Saxl graph is empty.
\end{proof}

\begin{corollary}
	Let $p=2^f-1\ge 7$ be a Mersenne prime, $G=\PGL_2(p)$ and $H=D_{2(p+1)}$. Then $\val(G,H)=0$.
	\label{corollary C3 PGL}
\end{corollary}

In view of Table \ref{table almost simple stabiliser 2-group}, we deduce that the proof of Theorem \ref{thm classification prime-power} is complete by combining Corollaries \ref{corollary C2 PSL}, \ref{corollary C2 PGL}, \ref{corollary C3 PSL} and \ref{corollary C3 PGL}.

\section{Odd Valencies}

\label{section odd valencies}

A rough classification of almost simple primitive groups with odd valency in \cite[Proposition 3.2]{TimSaxlGraph} gives the following proposition.

\begin{proposition}
	Let $G$ be an almost simple primitive group with stabiliser $H$ such that $\val(G,H)$ is odd. Then one of the following holds:
	\begin{enumerate}[{\indent\rm (i)}]
		\item $G=\M_{23}$ and $H=23{:}11$.
		\item $G=A_p$ and $H=\AGL_1(p)\cap G=\mathbb{Z}_p{:}\mathbb{Z}_{(p-1)/2}$, where $p$ is a prime such that $p\equiv 3\Mod 4$ and $(p-1)/2$ is composite.
		\item $\soc(G)=\LL_r^\epsilon(q)$, $H\cap\soc(G)=\mathbb{Z}_a{:}\mathbb{Z}_r$ and $G\ne\soc(G)$, where $a=\frac{q^r-\epsilon}{(q-\epsilon)(r,q-\epsilon)}$, $r\ge 3$ is a prime and $\soc(G)\ne \UU_3(3),\UU_5(2)$.
	\end{enumerate}
\label{prop odd val}
\end{proposition}

Note that the valency of $G=A_p$ with stabiliser $H=\AGL_1(p)\cap A_p$ is given in (\ref{equation val Ap}). It follows by this equation that $\val(G,H)$ is odd if and only if $2$ divides $\frac{p-1}{2}$. However, it makes $p\equiv 1\Mod 4$, a contradiction. Therefore, the second case in Proposition \ref{prop odd val} cannot happen and so there are only two possible cases given in Theorem \ref{thm odd valency rewrite}.

In this section we set $G=\PGL^\epsilon_r(q)$ for an odd prime $r$, and $H=\langle z\rangle{:}\langle \sigma\rangle\cong\mathbb{Z}_a{:}\mathbb{Z}_r$ a maximal subgroup of $G$, where $a=\frac{q^r-\epsilon}{q-\epsilon}$. If $H$ is a Frobenius group then Theorem \ref{thm Frobenius val} applies. In general, however, $H$ is not Frobenius but ``almost" Frobenius.

\begin{lemma}	\label{lemma frobenius c3}
	The group $H=\langle x\rangle\times(\langle y\rangle{:}\langle\sigma\rangle)\cong\mathbb{Z}_{(q-\epsilon,r)}\times F$, where $x=z^{a/(q-\epsilon,r)}$ and $F$ is a Frobenius group.
\end{lemma}
\begin{proof}
	First observe that
	\[(q-\epsilon,a)=(q-\epsilon,q^{r-1}+\epsilon q^{r-2}+q^{r-3}+\epsilon q^{r-4}+\cdots +\epsilon q+1)=(q-\epsilon,r),  \]
	and
	\[(q+\epsilon,a)=(q+\epsilon,q^{r-1}+\epsilon q^{r-2}+q^{r-3}+\epsilon q^{r-4}+\cdots +\epsilon q+1)=1.    \]
	
    Since $r$ is an odd prime, $(q-\epsilon,r)=1$ or $r$. Hence, we may divide our proof into two cases.
    
    If $(q-\epsilon,r)=1$, then $(q-\epsilon,a)=1$.
    Thus, for any element $z^i\in \langle z\rangle$, $(z^i)^\sigma=z^i$ implies that $z^{i(q-\epsilon)}=1$. Since $(q-\epsilon,a)=1$, $z^i=1$ and $\sigma$ induces a fixed-point-free automorphism of prime order on $\langle z\rangle$. 
    This implies that $H$ is a Frobenius group. 
    
    If $(q-\epsilon,r)=r=(q-\epsilon,a)$, set $\ell=q-\epsilon$. It follows that
    \begin{align*}
        \frac{a}{r}&=\frac{1}{r}\left(\frac{q^r-\epsilon}{q-\epsilon}\right)=\frac{1}{r}\left(\frac{(\ell+\epsilon)^r-\epsilon}{\ell}\right)\\
         &=\frac{1}{r}\left(\sum_{i=1}^{r}{\binom{r}{i}}\ell^{i-1}\epsilon^{r-i}\right)=\left(\sum_{i=2}^{r}{\binom{r}{i}}\frac{\ell^{i-1}\epsilon^{r-i}}{r}\right)+1\\ &=\left(\sum_{i=2}^{r-1}\ell^{i-2}\epsilon^{r-i}\frac{{\binom{r}{i}}}{r}\right)\cdot \ell+\frac{\ell}{r}\cdot\ell^{r-2}+1\\
         &\equiv 1\Mod{r}.
    \end{align*}
    This implies $(\frac{a}{r},r)=1$ and $H=(\langle z^{a/r}\rangle \times \langle z^r\rangle){:}\langle \sigma\rangle $. Furthermore, $(\frac{a}{r},q-\epsilon)$ divides $(\frac{a}{r},\frac{q-\epsilon}{r})\cdot(\frac{a}{r},r)=1$, which implies that $(\frac{a}{r},q-\epsilon)=1$.
    By an argument similar to the previous case, $\sigma$ induces a trivial automorphism on $\langle z^{a/r}\rangle$ and a fixed-point-free automorphism on $\langle z^r\rangle$. It follows that $H=\langle z^{a/r}\rangle \times (\langle z^r\rangle{:}\langle \sigma\rangle)=\mathbb{Z}_{(q-\epsilon,r)}\times F$, where $F$ is a Frobenius group.
\end{proof}

In the following lemma, we maintain the notation of Lemma \ref{lemma frobenius c3}.

\begin{lemma}
    \label{lemma order of preimage of x in GL}
    Suppose $(q-\epsilon,r)=r$. Then there exist preimages $\tilde{x},\tilde{\sigma}\in\GL^\epsilon_r(q)$ of $x$ and $\sigma$ respectively and a basis of $\mathbb{F}_q^r$ (or $\mathbb{F}_{q^2}^r$ for the case $\epsilon=-$) such that the matrices of $\tilde{x}$ and $\tilde{\sigma}$ under this basis are
    \begin{equation*}
        \begin{bmatrix}
        &1&&\\
        &&\ddots&\\
        &&&1\\
        \mu&&&
        \end{bmatrix} \hspace{2mm}\mbox{and}\hspace{2mm}
        \begin{bmatrix}
        1&&&\\
        &\lambda&&\\
        &&\ddots&\\
        &&&\lambda^{r-1}
        \end{bmatrix}
    \end{equation*}
    respectively, where $|\mu|=q-\epsilon$, $\lambda=\mu^{(q-1)/r}$ when $\epsilon=+$ and $\lambda=\mu^{(q^2-1)/r}$ when $\epsilon=-$. Furthermore, we have that $\tilde{x}$ is conjugate with $\tilde{x}\tilde{\sigma}^i$ in $\GL^\epsilon_r(q)$, where $1\leqslant i\leqslant r-1$.
\end{lemma}

\begin{proof}
 
 If $\epsilon=+$, then let $V=\mathbb{F}_{q^r}$. We can also view $V\cong \mathbb{F}_q^{r}$ as an $\mathbb{F}_q$-linear space with additional field structure. For any $b\in \mathbb{F}_{q^r}$, let $\pi_b: V\rightarrow V$ be the multiplication of $b$. That is, $\pi_b(x)=bx$ for all $x\in V$. Let $\tau$ be the field automorphism of $\mathbb{F}_{q^r}=V$ which maps $x$ to $x^q$ for all $x\in V$. It is easy to see that both $\pi_b$ and $\tau$ are $\mathbb{F}_q$-linear transformations of $V$. 
By \cite[Section 4.3]{KleidmanLiebeckClassicalGroups}, there exist $\omega\in \mathbb{F}_{q^r}$ and a homomorphism $\phi$ from $\GL(V)$ to $\PGL_r(q)$ such that $\phi(\pi_{\omega})=x$, $\phi(\tau)=\sigma$ and $\ker \phi=Z(\GL(V))\cong \mathbb{F}_{q}^{*}$. Note that 
 % $|\omega_1|=|\pi_{\omega_1}|$ divide $|\ker (\phi)|\cdot|x|=(q_1)r$. If $|\omega_1|$
 $\phi^{-1}(\langle x\rangle)=(\ker\phi)\langle \pi_\omega\rangle$ is a cyclic group of order $r(q-1)$. We may assume that $\pi_\omega$ is a generator of  $\phi^{-1}(\langle x\rangle)$. It follows that $|\omega|=r(q-1)$. Set $\mu=\omega^r\in \mathbb{F}_q$ and $\lambda=\omega^{q-1}=\mu^{\frac{q-1}{r}}$. We have $|\mu|=q-1$. 
 Now set $\tilde{x}=\pi_\omega$ and $\tilde{\sigma}=\tau$. 
 As $\omega\in \mathbb{F}_{q^r}\setminus \mathbb{F}_q$ and $r$ is a prime, we have $\mathbb{F}_{q^r}=\mathbb{F}_{q}(\omega)$. This implies that $1,\omega, \dots, \omega^{r-1}$ form a basis of $V$. It is easy to see that the matrices of $\tilde{x}$ and $\tilde{\sigma}$ under this basis are as desired. Furthermore, for $1\leqslant i \leqslant r-1$, let $g_i\in \GL(V)$ be such that the matrix of $g_i$ under the basis described above is 
\begin{equation}\label{trans matrix}
\begin{bmatrix}1&&&&\\        &\lambda^i&&&\\		 &&\lambda^{i(1+2)}&&\\        &&&\ddots&\\        &&&&\lambda^{i(1+2+\dots +(r-1))}        \end{bmatrix}.
\end{equation}
A simple computation gives $\tilde{x}^{g_i}=\tilde{x}\tilde{\sigma}^i$. 

 If $\epsilon=-$, then let $V=\mathbb{F}_{q^{2r}}$ with a Hermitian form $\beta$ such that $\beta(x,y)=xy^{q^r}$. We can also view $V\cong \mathbb{F}_{q^2}^{r}$ as an $\mathbb{F}_{q^2}$-linear space with additional field structure. 
 Let $T$ be the trace map from $\mathbb{F}_{q^{2r}}$ to $\mathbb{F}_{q^2}$, and $\beta^*=T\beta$ be a Hermitian form over $\mathbb{F}_{q^2}$. That is, $(V,\beta)$ is a 1-dimensional unitary space, while $(V,\beta^*)$ is an $r$-dimensional unitary space.
 For any $b\in \mathbb{F}_{q^r}$, define $\pi_b: V\rightarrow V$ by $x\mapsto bx$, which is the multiplication of $b$. Let $\tau$ be the field automorphism of $\mathbb{F}_{q^r}=V$ which maps $x$ to $x^{q^2}$ for all $x\in V$, and $c$ be an element of order $q^r+1$ in $\mathbb{F}_{q^{2r}}^*$. 
 It is easy to see that both $\pi_c$ and $\tau$ belong to $\GU_r(q)$. Again, by \cite[Section 4.3]{KleidmanLiebeckClassicalGroups}, there exists a homomorphism $\phi$ from $\GU_r(q)$ to $\PGU_r(q)$ such that $\phi(\pi_{b})=z$ and $\phi(\tau)=\sigma$. Set $\omega=c^{a/r}$ and $\mu=\omega^r$. It follows that $\mu\in \mathbb{F}_{q^2}$. 
   Since $\mathbb{F}_{q^{2r}}=\mathbb{F}_{q^2}(\omega)$, the $r$ elements $1,\omega, \dots, \omega^{r-1}$ form an $\mathbb{F}_{q^2}$-basis of $V$. Set $\tilde{x}=\pi_\omega$ and $\tilde{\sigma}=\tau$.
   It is easy to see that the matrices of $\tilde{x}$ and $\tilde{\sigma}$ under this basis are as desired. Furthermore, for $1\leqslant i \leqslant r-1$ we have $\tilde{x}^{g_i}=\tilde{x}\tilde{\sigma}^i$, where $g_i\in \GU_r(q)$ is the matrix defined in (\ref{trans matrix}), which completes the proof. 
\end{proof}

Suppose $(q-\epsilon,r)=r$. By Lemma \ref{lemma frobenius c3}, the subgroups of order $r$ in $H$, up to conjugacy in $H$, have representatives $\{\langle x\rangle, \langle x\sigma\rangle,\dots,\langle x\sigma^{r-1}\rangle,\langle\sigma\rangle\}$. Lemma \ref{lemma order of preimage of x in GL} implies that even if they are not conjugate in $H$, $\langle x\rangle$ and $\langle x\sigma^i\rangle$ are conjugate in $G$. Therefore, there are only two conjugacy classes in $G$ of subgroups of order $r$ in $H$, which have representatives $\{\langle x\rangle,\langle \sigma\rangle\}$.

\begin{corollary}\label{corollary normaliser r^2}
    We have $|N_{G}(\langle x,\sigma\rangle)|=r^3$ if $(q-\epsilon,r)=r$.
\end{corollary}

\begin{proof}
    There are $r+1$ subgroups of $\langle x,\sigma\rangle$ of order $r$, namely $\langle x\rangle, \langle x\sigma\rangle,\dots,\langle x\sigma^{r-1}\rangle$ and $\langle\sigma\rangle$. The image of matrix $g_i$ in the proof of Lemma \ref{lemma order of preimage of x in GL} maps $\langle x\rangle$ to $\langle x\sigma^i\rangle$ and fixes $\langle \sigma\rangle$ by conjugation. It follows that there are two orbits of $N_{G}(\langle x,\sigma\rangle)$ on the $r+1$ subgroups. They are $\{\langle x\rangle, \langle x\sigma\rangle,\dots,\langle x\sigma^{r-1}\rangle\}$ and $\{\langle\sigma\rangle\}$. The stabiliser of the former is $\langle x,\sigma\rangle$, which has order $r^2$. Therefore, $N_{G}(\langle x,\sigma\rangle)$ is of order $r^3$.
\end{proof}

To obtain the valency, we need to determine $|N_G(\langle\sigma\rangle)|$, which is denoted by $N$ for convenience. Indeed, $N=(r-1)|C_G(\langle\sigma\rangle)|$ and the order of the centralisers are determined in \cite{TimMichaelClassicalGroups}. More specifically, the first three rows in \cite[Table B.3]{TimMichaelClassicalGroups} give the linear case and the first five rows in \cite[Table B.4]{TimMichaelClassicalGroups} give the unitary case. To see this, observe that $\sigma$ is conjugate to the matrix
\begin{equation*}
    \begin{bmatrix}
    &1&&\\
        &&\ddots&\\
        &&&1\\
        1&&&
    \end{bmatrix}
\end{equation*}
and so if $(q,r)=r$ it is a regular unipotent element,
while otherwise it has $r$ distinct eigenvalues in a suitable field extension of $\mathbb{F}_q$. In particular, Lemma \ref{lemma order of preimage of x in GL} applies when $(q-\epsilon,r)=r$, making $\sigma$ diagonalisable over $\mathbb{F}_q$ (or $\mathbb{F}_{q^2}$ for the unitary case). In this case, the third row in \cite[Table B.3]{TimMichaelClassicalGroups} and the fifth row in \cite[Table B.4]{TimMichaelClassicalGroups} apply. This gives the following lemma on $N=|N_G(\langle\sigma\rangle)|$.

\begin{lemma}
\label{lemma N}
    The following statements hold.
    \begin{enumerate}[\rm (i)]
        \item If $(q,r)=r$ then $N=(r-1)q^{r-1}$.
        \item If $(q-\epsilon,r)=r$ then $N=r(r-1)(q-\epsilon)^{r-1}$.
        \item Otherwise,
        		\begin{equation*}
		N=
		\begin{cases}
		(r-1)(q^k-1)^{\frac{r-1}{k}} & \mbox{ if } \epsilon=+ \mbox{, or }\epsilon=- \mbox{ and } k \equiv 0\Mod{4},\\ 
		(r-1)(q^{k/2}+1)^{\frac{2r-2}{k}} & \mbox{ if } \epsilon=- \mbox{ and } k \equiv 2\Mod{4},\\ 
		(r-1)(q^{2k}-1)^{\frac{r-1}{2k}} & \mbox{ if } \epsilon=- \mbox{ and } k \mbox{ odd},\\ 
		\end{cases} 
		\end{equation*}
		where $k$ is the smallest integer such that $r\mid q^k-1$.
    \end{enumerate}
\end{lemma}

%\subsection{Valencies and Multiplicities of Subdegrees}

Now we are ready to obtain the valencies as well as the multiplicities of subdegrees.

\begin{proposition} \label{prop subdegree PGLepsilon C3 (q-epsilon,r)=1}
    If $(q-\epsilon,r)=1$, then subdegrees of $G$ with stabiliser $H$ are $1$, $|H|$, $|H|/r$
    with multiplicities
    \begin{equation*}
        1,\ \val(G,H)/|H|,\ (N-r)/r
    \end{equation*}
    respectively, where
    \begin{equation*}
        \val(G,H)=|G:H|+a-1-\frac{aN}{r}.
    \end{equation*}
\end{proposition}

%     \begin{table}[ht]
% 	\centering
% 	% \renewcommand\arraystretch{2}
% 	\begin{tabular}{cc}
% 		\Xhline{2pt}
% 		Subdegrees&Multiplicities\\
% 		\Xhline{1pt}
% 		$1$&$1$\\
% 		$|H|$&$\val(G,H)/|H|$\\
% 		$m$&$(N-r)/r$\\
% 		\Xhline{2pt}
% 	\end{tabular}
% 	\caption{Subdegrees and multiplicities of $G$ with stabiliser $H$ when $(q-\epsilon,r)=1$}
% 	\label{table PGLepsilon C3 subdegree (q-epsilon,r)=1}
% \end{table}

\begin{proof}
    Combine Theorem \ref{thm Frobenius val} and Lemma \ref{lemma frobenius c3}.
\end{proof}

\begin{proposition}
    If $(q-\epsilon,r)=r$, then subdegrees of $\soc(G)=\LL_r^\epsilon(q)$ with stabiliser $H\cap\LL_r^\epsilon(q)$ are $1$, $a$, $a/r$
    with multiplicities
    \begin{equation*}
        1,\ \val(\LL_r^\epsilon(q),H\cap\LL_r^\epsilon(q))/a,\ (r-1)(q-\epsilon)^{r-1}-r
    \end{equation*}
    respectively, where
    \begin{equation*}
        \val(\LL_r^\epsilon(q),H\cap\LL_r^\epsilon(q))=|G:H|+\frac{|H|}{r}-1-\frac{|H|}{r^2}(r-1)(q-\epsilon)^{r-1}.
    \end{equation*}
\end{proposition}

\begin{proof}
    Note that $|G:H|=|\LL_r^\epsilon(q):H\cap \LL_r^\epsilon(q)|$. Again, the result follows by Theorem \ref{thm Frobenius val}, Lemmas \ref{lemma frobenius c3} and \ref{lemma N}.
\end{proof}

\begin{proposition}
	Suppose $(q-\epsilon,r)=r$. Then subdegrees of $G$ with stabiliser $H$ are $1$, $|H|$, $a$, $a/r$ with multiplicities
	\begin{equation*}
	    1,\ \val(G,H)/|H|,\ m(G,H,a),\ r-1
	\end{equation*}
	respectively, where
	\begin{equation*}
	    %\val(G,H)=|G:H|-\left(r-1\right)\frac{m}{r}\left(2-r+m-\frac{m}{r}\right)-\left(\frac{m}{r}\right)^2\cdot\frac{N_{r,q}-r^2}{|H|}+\left(r-1\right)\frac{m}{r}-1
	    \val(G,H)=|G:H|-\frac{|H|}{r^3}\left(r-1\right)\left((r-1)\left(\frac{|H|}{r}-r\right)+(q-\epsilon)^{r-1}\right)+\frac{|H|}{r^2}-1
	\end{equation*}
	and
	\begin{equation*}
	    m(G,H,a)=\frac{1}{r^3}(r(r-1)(q-\epsilon)^{r-1}-r^4+r^3-r^2+(r-1)^2|H|).
	\end{equation*}
% 	Moreover, subdegrees of $\soc(G)=\LL_r^\epsilon(q)$ with stabiliser $H\cap\LL_r^\epsilon(q)$ are shown in Table \ref{table PSL C3 subdegree (q-1,r)=r}, where
% 	\begin{equation*}
% 	    \val(\LL^\epsilon_r(q),H\cap\LL^\epsilon_r(q))=|G:H|+\frac{|H|}{r}-1-\frac{|H|}{r^2}(r-1)(q-\epsilon)^{r-1}.
% 	\end{equation*}

	\label{proposition val PGL C3 n prime}
\end{proposition}

% \begin{table}[ht]
% 	\centering
% 	% \renewcommand\arraystretch{2}
% 	\begin{tabular}{cc}
% 		\Xhline{2pt}
% 		Subdegrees&Multiplicities\\
% 		\Xhline{1pt}
% 		$1$&$1$\\
% 		$|H|$&$\val(G,H)/|H|$\\
% 		$|H|/r$&$\frac{1}{r^3}(r(r-1)(q-\epsilon)^{r-1}-r^4+r^3-r^2+(r-1)^2|H|)$\\
% 		$|H|/r^2$&$r-1$\\
% 		\Xhline{2pt}
% 	\end{tabular}
% 	\caption{Subdegrees and multiplicities of $G$ with stabiliser $H$ when $(q-\epsilon,r)=r$}
% 	\label{table PGL C3 subdegree (q-1,r)=r}
% \end{table}

% \begin{table}[ht]
% 	\centering
	% \renewcommand\arraystretch{2}
% 	\begin{tabular}{cc}
% 		\Xhline{2pt}
% 		Subdegrees&Multiplicities\\
% 		\Xhline{1pt}
% 		$1$&$1$\\
% 		$|H|/r$&$r\cdot \val(\LL^\epsilon_r(q),H\cap\LL^\epsilon_r(q))/|H|$\\
% 		$|H|/r^2$&$(r-1)(q-\epsilon)^{r-1}-r$\\
% 		\Xhline{2pt}
% 	\end{tabular}
% 	\caption{Subdegrees and multiplicities of $\LL_r^\epsilon(q)$ with stabiliser $H\cap\LL_r^\epsilon(q)$ when $(q-\epsilon,r)=r$}
% 	\label{table PSL C3 subdegree (q-1,r)=r}
% \end{table}

\begin{proof}
	First we find the possible arc stabilisers of $G$ with point stabiliser $H$. If a non-identity element $h\in H\cap H^g\ne H$ then $h$ is of order $r$, otherwise $g\in N_G(\langle h\rangle)=H$, which leads to a contradiction. This implies that an element in $\cali$ has order $1$, $r$, $r^2$ or $|H|$. Moreover, if $x\in H\cap H^g\ne H$ for some $g\in G$ then $[x,x^g]=1$ as $Z(H^g)=\langle x^g\rangle$. Thus, we have $x^g\in C_G(\langle x\rangle)=H$, which gives $x^g\in H\cap H^g$, while $x\ne x^g$. It follows that $\langle x\rangle$ cannot be an arc stabiliser. With this in mind and by Lemma \ref{lemma order of preimage of x in GL}, $\cali$ can be chosen as
	\begin{equation*}
	    \cali=\{1,\langle \sigma\rangle,\langle x\sigma\rangle,\langle x^2\sigma\rangle,\dots,\langle x^{r-1}\sigma\rangle,\langle x,\sigma\rangle,H\}
	\end{equation*}
	and a set of representatives of the equivalent relation (E1)-(E3) defined in Section \ref{section Strategy} is
	\begin{equation*}
	    \tilde{\cali}=\{1,\langle\sigma\rangle,\langle x\sigma\rangle,\langle x,\sigma\rangle,H\}.
	\end{equation*}
	Now the matrix of $\tilde{\eta}$ follows that
	\begin{equation*}
	    M=
	    \begin{bmatrix}
	    1&\frac{|H|}{r^2}&\frac{(r-1)|H|}{r^2}&\frac{|H|}{r^2}&1\\
	    0&1&0&1&1\\
	    0&0&1&1&1\\
	    0&0&0&1&1\\
	    0&0&0&0&1
	    \end{bmatrix}
	\end{equation*}
	and so
	\begin{equation*}
	M^{-1}=
	\begin{bmatrix}
	1&-\frac{|H|}{r^2}&-\frac{(r-1)|H|}{r^2}&\frac{(r-1)|H|}{r^2}&-1+\frac{|H|}{r^2}\\
	&1&0&-1&0\\
	&&1&-1&0\\
	&&&1&-1\\
	&&&&1
	\end{bmatrix}.
	\end{equation*}
	To calculate the values of $\Delta$, Corollary \ref{corollary normaliser r^2} and Lemma \ref{lemma N} can be applied. Note also that $$\mathcal{S}\cap \langle x\sigma\rangle=\{\langle x\rangle,\langle x\sigma\rangle,\langle x^2\sigma\rangle,\dots,\langle x^{r-1}\sigma\rangle\}$$
	and $\mathcal{S}\cap A^G=\{A\}$ for any other $A\in\tilde{\cali}$ by Lemma \ref{lemma order of preimage of x in GL}. We finally obtain that
	\begin{equation*}
	    \begin{bmatrix}
    |\Delta(1)|\\
    |\Delta(\langle\sigma\rangle)|\\
    |\Delta(\langle x\sigma\rangle)|\\
    |\Delta(\langle x,\sigma\rangle)|\\
    |\Delta(H)|
    \end{bmatrix}
    =
    \begin{bmatrix}
    |G|\\
    N|H|/r^2\\
    |H|^2\cdot (\frac{1}{|H|}+\frac{r-1}{r^2})\\
    |H|\cdot r^3/r^2\\
    |H|
    \end{bmatrix}
    =
    \begin{bmatrix}
    |G|\\
    |H|\cdot (r-1)\cdot (q-\epsilon)^{r-1}/r\\
    |H|+\frac{(r-1)|H|^2}{r^2}\\
    r|H|\\
    |H|
    \end{bmatrix}.
	\end{equation*}
	Now the valencies and multiplicities of subdegrees follow by the matrix operations given in Remark \ref{remark matrix form of the main method}.
\end{proof}

It is straightforward to deduce Theorem \ref{thm odd valency rewrite} by combining Theorem \ref{thm table val almost simple soluble stabiliser alternating}, Propositions \ref{prop odd val}, \ref{prop subdegree PGLepsilon C3 (q-epsilon,r)=1}-\ref{proposition val PGL C3 n prime}.

\bibliographystyle{abbrv}
\bibliography{ValSaxl_bib}

\end{document}